\documentclass[a4paper,11pt]{article}

\usepackage[
    a4paper,
    top=1.05in,
    bottom=1.05in,
    left=1.00in,
    right=1.00in
]{geometry}
\usepackage[T1]{fontenc}
\usepackage{amsmath,amsthm,mathtools}
\usepackage{amssymb}
\usepackage{newtxtext,newtxmath}
\usepackage{microtype}
\microtypesetup{protrusion=true,expansion=true}

\usepackage{mathrsfs}
\usepackage{slashed}
\usepackage{upgreek}
\allowdisplaybreaks[4]
\usepackage{enumitem}
\setlist{itemsep=0.25em, topsep=0.45em, parsep=0pt}
\usepackage{scalerel}
\usepackage[usestackEOL]{stackengine}
\usepackage[thinc]{esdiff}

\usepackage[all]{xy}
\usepackage{graphicx,psfrag}
\usepackage{pstool}
\usepackage{tikz-cd}
\usepackage[font=small,labelfont=bf,labelsep=period]{caption}

\usepackage{titlesec}
\titleformat{\section}{\normalfont\Large\bfseries}{\thesection}{0.75em}{}
\titleformat{\subsection}{\normalfont\large\bfseries}{\thesubsection}{0.75em}{}
\titlespacing*{\section}{0pt}{2.6ex plus 0.8ex minus 0.2ex}{1.2ex plus 0.2ex}
\titlespacing*{\subsection}{0pt}{2.0ex plus 0.6ex minus 0.2ex}{0.9ex plus 0.2ex}

\usepackage{fancyhdr}
\fancypagestyle{plain}{%
  \fancyhf{}%
  \fancyfoot[C]{\thepage}%
}

\usepackage[dvipsnames]{xcolor}
\definecolor{linkblue}{RGB}{20,67,128}
\definecolor{citegreen}{RGB}{25,105,80}
\definecolor{urlpurple}{RGB}{100,55,125}

\usepackage[colorlinks=true,
            linkcolor=linkblue,
            citecolor=citegreen,
            urlcolor=urlpurple,
            plainpages=false]{hyperref}
\usepackage{bookmark}

\usepackage{array,tabularx,booktabs}

 \usepackage{setspace}

\usepackage[titletoc,title]{appendix}
\usepackage{titletoc}

\usepackage[capitalize,nameinlink,noabbrev]{cleveref}

\newcommand{\GHconvtext}{\mathrm{Gromov\text{--}Hausdorff}}
\newcommand{\ucite}[2][]{\textup{\cite[#1]{#2}}}

 \newcommand{\N}{\ensuremath{\mathbb{N}}}
 
 \newcommand{\R}{\ensuremath{\mathbb{R}}}
 
 \def\d{\mathrm{d}}
  \def\III{\mathbb{I}}
 
 \def\RR{\mathcal{R}}
 \def\MM{\mathcal{M}}

  \def\cc{\mathfrak{c}}

 \def\Var{\mathrm{Var}}

 \def\scal{\mathrm{R}}

  \def\diam{\mathrm{diam}}

 \newcommand{\RP}{\ensuremath{\mathbb{RP}}}

 \newcommand{\ba}{\begin{align*}}
 \newcommand{\ea}{\end{align*}}

\newcommand{\lc}{\left(}
\newcommand{\rc}{\right)}
 
\newcommand{\ep}{\epsilon}

\def\HHH{\mathscr{H}}%use \AAA to denote \mathscr{A}

\def\CCC{\mathscr{C}}

\def\NN{\mathcal{N}}
\def\NNN{\mathscr{N}}

\newcommand{\Rm}{\ensuremath{\mathrm{Rm}}}
\newcommand{\Ric}{\ensuremath{\mathrm{Ric}}}

\renewcommand{\t}{\mathfrak{t}}

\def\MS{\mathcal{S}}

\DeclareMathOperator{\GH}{GH}
 \makeatletter
 \def\ExtendSymbol#1#2#3#4#5{\ext@arrow 0099{\arrowfill@#1#2#3}{#4}{#5}}
 
 \makeatother

 \makeatletter
 \def\ExtendSymbol#1#2#3#4#5{\ext@arrow 0099{\arrowfill@#1#2#3}{#4}{#5}}
 \newcommand\longright[2][]{\ExtendSymbol{-}{-}{\rightarrow}{#1}{#2}}
 \makeatother
 
\def\aint{\,\ThisStyle{\ensurestackMath{%
  \stackinset{c}{.2\LMpt}{c}{.5\LMpt}{\SavedStyle-}{\SavedStyle\phantom{\int}}}%
  \setbox0=\hbox{$\SavedStyle\int\,$}\kern-\wd0}\int}

\DeclarePairedDelimiter\abs{\lvert}{\rvert}%
\makeatletter
\let\oldabs\abs
\def\abs{\@ifstar{\oldabs}{\oldabs*}}
\makeatother

\numberwithin{equation}{section}
\theoremstyle{plain}
\newtheorem{thm}{Theorem}[section]
\newtheorem{cor}[thm]{Corollary}
\newtheorem{prop}[thm]{Proposition}
\newtheorem{lem}[thm]{Lemma}

\theoremstyle{definition}
\newtheorem{defn}[thm]{Definition}

\theoremstyle{remark}
\newtheorem{rem}[thm]{Remark}

\titlecontents{part}[0pt]  % 作用于part条目
{\addvspace{10pt}\bfseries} % 条目前格式（粗体+垂直间距
{\makebox[3em][l]{PART-\thecontentslabel}}      % 标签（如“Part I”）
{}                          % 无编号时的占位符
{}                          % 页面数字格式
\title{Gromov--Hausdorff convergence of time-slices of singular Ricci flows in dimension three}
\author{Yu Li} 
\date{\today}

\begin{document}
	%\begin{CJK}{UTF8}{gbsn}

\maketitle

\begin{abstract}
Starting from a closed, orientable three-dimensional Riemannian manifold, we consider the completion of the associated singular Ricci flow with respect to a natural spacetime distance. We show that this completion admits canonical intrinsic metrics on its time-slices, defined by conjugate heat kernel measures, and that these time-slices arise as metric limits of the regular part of the flow.

More precisely, for any $t_0>0$ and any connected component $Z_{t_0}'$ of the time-slice $Z_{t_0}$ of the completion, we prove that
\[
(\overline{\mathcal R_t'},d_{g_t}) \xrightarrow[t\nearrow t_0]{\GHconvtext} (Z_{t_0}',d_{t_0}^Z),
\]
where $\mathcal R_t'$ is the corresponding connected component of the regular part and $\overline{\mathcal R_t'}$ denotes its metric completion. In particular, this yields the Gromov--Hausdorff convergence at the first singular time for closed three-dimensional Ricci flows.

We also establish a refined structure theory for the singular set of the completion. In particular, the singular set is horizontally parabolic $1$-rectifiable, and its time image has vanishing \(1/2\)-dimensional Hausdorff measure. Moreover, on each time-slice, the singular set has Minkowski dimension at most~$1$.

The proof relies on heat kernel estimates for singular Ricci flows and the generalization of the structure theory of noncollapsed Ricci flow limit spaces.
\end{abstract}

\tableofcontents

\section{Introduction}
Ricci flow, introduced by Hamilton \cite{hamilton1982three}, is one of the central tools in geometric analysis and three-dimensional topology. Starting from a Riemannian manifold \((M,g_0)\), the Ricci flow evolves the metric by
\[
    \partial_t g_t=-2\Ric(g_t).
\]
In dimension three, finite-time singularities are unavoidable in general. Hamilton--Perelman's theory of Ricci flow with surgery \cite{hamilton1997,perelman2003surgery,perelman2003finite} and the later theory of singular Ricci flows developed by Kleiner--Lott and Bamler--Kleiner \cite{kleinerlott2017singular,bamlerkleiner2022} provide canonical ways to continue the flow through singularities.

The purpose of this paper is to study the metric behavior of the time-slices of a three-dimensional singular Ricci flow as one approaches a fixed time. More precisely, we are interested in the following question.

\medskip

\noindent
\textit{Does a singular Ricci flow time-slice admit a canonical Gromov--Hausdorff limit as it approaches a fixed time, and if so, how is this limit related to the associated singular Ricci flow?}

\medskip

For a smooth closed Ricci flow approaching its first singular time $T$, one might hope that the metric spaces \((M,d_{g_t})\) converge as~\(t\nearrow T\). A priori, this is not clear. Indeed, in the absence of a uniform lower Ricci curvature bound, one cannot directly apply Gromov's compactness theorem. Even if sequential Gromov--Hausdorff limits exist, different sequences \(t_j\nearrow T\) could, in principle, lead to different limits. The first goal of this paper is to show that, in dimension three, this ambiguity does not occur.

Let
\[
    (M^3,(g_t)_{t\in[0,T)})
\]
be a closed Ricci flow with normalized initial condition; see Definition~\ref{def:normal}. We assume that \(T<\infty\) is the first singular time. In our previous work \cite{fangli2025recti}, we proved an \(L^1\)-curvature bound in dimension four and used it to resolve Perelman's bounded diameter conjecture in dimension three. In particular, for the present three-dimensional flow one has
\[
    \sup_{t\in[0,T)}\left(
        \int_M |\Rm|\,\d V_{g_t}+\diam_{g_t}(M)
    \right)<\infty.
\]
Using this crucial estimate, we consider the measure
\[
    \mu_t(A)=\int_A (|\scal|+1)\,\d V_{g_t}, \qquad t\in[0,T),
\]
and show in Proposition~\ref{prop:ballpacking} that
\[
    \mu_t(B_{g_t}(x,\ep))\ge c_\ep
\]
for every ball \(B_{g_t}(x,\ep)\neq M\), where \(c_\ep>0\) depends only on \(\ep\) and the Ricci flow, but is independent of~\(t\).

The proof is based on the canonical neighborhood theorem for the high-curvature region of a three-dimensional Ricci flow. A key observation is that, on a \(\delta\)-tube, namely a region in which every point is the center of a \(\delta\)-neck, the contribution of the scalar curvature integral is comparable to the diameter of the region. On the other hand, \(\mu_t(M)\) is uniformly bounded by the estimates above. Thus, a standard ball-packing argument implies that any sequence \(t_j\nearrow T\) admits a subsequence such that \((M,d_{g_{t_j}})\) converges in the Gromov--Hausdorff sense to a compact metric space. The main remaining issue is to identify this limit and prove that it is independent of the sequence.

To do so, we use the spacetime distance introduced in \cite{fangli2025structure}. More precisely, we equip \(M\times[1,T)\) with a distance \(d^*\), defined in terms of conjugate heat kernel measures and the \(W_1\)-Wasserstein distance; see \eqref{defndstar} for the precise definition. Let
\[
    (Z,d_Z,\mathfrak t)
\]
be the metric completion of \(M\times[1,T)\) with respect to \(d^*\), where \(\mathfrak t\) denotes the time function. As in \cite[Section 9]{fangli2025structure}, the completion adds only points at time \(T\). We write
\[
    Z_T:=\mathfrak t^{-1}(T).
\]
Following \cite{fangli2025structure}, the slice \(Z_T\) carries a canonical intrinsic distance \(d_T^Z\), defined by
\[
    d_T^Z(x,y)
    :=
    \lim_{s\nearrow T}
    d^{g_s}_{W_1}(\nu_{x;s},\nu_{y;s}),
\]
where \(\nu_{x;s}\) and \(\nu_{y;s}\) are the conjugate heat kernel measures based at \(x\) and \(y\), respectively.

As discussed in \cite{fangli2025structure}, there is some flexibility in the construction of the completion \((Z,d_Z,\mathfrak t)\). One may choose a different lower cutoff time \(\sigma\in(0,T)\), rather than \(1\), and define the spacetime distance \(d^*\) on \(M\times[\sigma,T)\). This may lead to a different completion \((Z,d_Z)\). Nevertheless, all such choices lead to the same terminal time-slice \(Z_T\), equipped with the same canonical intrinsic distance \(d_T^Z\); this distance is defined solely from the conjugate heat kernel measures and the spatial metrics on the regular time-slices (see \cite[Remark 6.2]{fangli2025structure}).

Our first main theorem identifies this intrinsic terminal slice with the Gromov--Hausdorff limit of the original Ricci flow.

\begin{thm}[Gromov--Hausdorff convergence at the first singular time]
Let \((M^3,(g_t)_{t\in[0,T)})\) be a closed Ricci flow with normalized initial condition, and suppose that \(T<\infty\) is the first singular time. Then
\[
    (M,d_{g_t})
    \xrightarrow[t\nearrow T]{\GHconvtext}
    (Z_T,d_T^Z).
\]
In particular, the Gromov--Hausdorff limit at the first singular time is unique.
\end{thm}

The proof has two main steps. First, using the \(L^1\)-curvature bound and the bounded diameter theorem, we obtain sequential Gromov--Hausdorff compactness of the time-slices. Given any sequential limit \((X_{\mathrm{GH}},d_{\mathrm{GH}})\) obtained from \((M,d_{g_{t_j}})\), we then construct an isometric embedding; see Proposition~\ref{prop:embh},
\[
    \iota:(Z_T,d_T^Z)\longrightarrow (X_{\mathrm{GH}},d_{\mathrm{GH}}),
\]
by approximating points of \(Z_T\) using \(H\)-centers at times \(t_j\nearrow T\). The key point is to prove that this embedding is surjective; see Proposition~\ref{prop:surj}. This is achieved through a quantitative covering argument based on the neck decomposition theorem from \cite{fangli2025recti}. Roughly speaking, every ball of definite radius in \((M,g_{t_j})\) must contain an \(H\)-center of some point in \(Z_T\); see Lemma~\ref{lem:key}. This rules out the possibility that a point in the sequential Gromov--Hausdorff limit lies outside the image of \(Z_T\).

We next pass from smooth Ricci flows to general singular Ricci flows. Let
\[
    (\mathcal M,\mathfrak t,\partial_{\mathfrak t},g)
\]
be the singular Ricci flow starting from a normalized, closed, orientable three-dimensional Riemannian manifold. This is a Ricci flow spacetime in the sense of Definition~\ref{def_RF_spacetime}. Then we consider the spacetime region \(\mathcal M_{[1,+\infty)}\) and define the spacetime distance \(d^*\) using the heat kernel on the singular Ricci flow; see Definition~\ref{defndstardistance}.

Recall that the singular Ricci flow \(\mathcal M\) is obtained as the limit of a sequence of Ricci flows with surgery as the surgery parameter tends to zero. Therefore, \(\mathcal M\) itself is not automatically a noncollapsed Ricci flow limit space in the sense of \cite[Definition 3.21]{fangli2025structure}, since the latter is defined as a Gromov--Hausdorff limit of closed Ricci flows. Consequently, one must first establish the relevant analytic estimates, such as heat kernel estimates, directly on the singular Ricci flow. These estimates are needed in order to extend the framework of \cite{fangli2025structure,fangli2025loja,fangli2025recti} to the present setting.

A substantial part of the paper, namely Section~\ref{sec:4}, is devoted to establishing these analytic estimates. The heat kernel on a singular Ricci flow was constructed by Lai \cite{lai2021}. We prove several important estimates for this heat kernel that parallel the corresponding estimates for smooth closed Ricci flows, including maximum principles, gradient estimates, Poincar\'e and logarithmic Sobolev inequalities, Gaussian concentration estimates, and monotonicity formulas for the pointed Nash and \(\mathcal W\)-entropies.

Most proofs are similar to those in the closed case. The main new issue is that, for a singular Ricci flow, each time-slice may be noncompact and incomplete. Thus, in many places, one must justify integrations by parts carefully. Thanks to the canonical neighborhood assumption, one can approximate a time-slice by regions whose boundary components are central two-spheres of \(\ep\)-necks. Using the rapid decay of the conjugate heat kernel near these boundary components, see Lemma~\ref{lem:potent1}, the corresponding boundary integrals vanish in the limit.

These estimates allow us to define the spacetime distance \(d^*\), extend conjugate heat kernel measures to the completion \(Z\), and analyze the intrinsic metrics \(d_t^Z\) on time-slices. Let
\[
    (Z,d_Z,\mathfrak t)
\]
be the metric completion of \(\MM_{[1,\infty)}\) with respect to \(d^*\). We denote the regular part by
\[
    \mathcal R:=\MM_{[1,\infty)},
\]
and the singular set by
\[
    \mathcal S:=Z\setminus \mathcal R.
\]
For each time \(t \ge 1\), we write
\[
    Z_t:=\mathfrak t^{-1}(t),\qquad
    \mathcal R_t:=\mathcal R\cap Z_t,\qquad
    \mathcal S_t:=\mathcal S\cap Z_t.
\]
As in the closed case, each slice \(Z_t\) carries a canonical intrinsic extended distance \(d_t^Z\), again defined by conjugate heat kernel measures; see Definition~\ref{defntimeslicedist}.

Let us briefly describe the geometric picture behind the proof. The spacetime completion \((Z,d_Z,\mathfrak t)\) is a parabolic metric space: the time function satisfies
\[
    |\mathfrak t(x)-\mathfrak t(y)|\le d_Z(x,y)^2.
\]
Near regular points, \(Z\) is modeled on the smooth Ricci flow spacetime \(\mathcal R\). Near singular points, blow-ups are modeled on quotient cylindrical Ricci shrinkers; see Subsection~\ref{subsec:tan}. Once the necessary analytic estimates are established, one can follow the same argument as in \cite{fangli2025recti} to obtain the neck decomposition theorem; see Theorem~\ref{neckdecomgeneral3dsing}. This theorem organizes high-curvature regions into cylindrical neck regions and controlled error regions.

Using the same method as in the proof of the Gromov--Hausdorff convergence at the first singular time, we obtain our second main result. It shows that the intrinsic metric slice \((Z_{t_0},d_{t_0}^Z)\) is the Gromov--Hausdorff limit of the corresponding regular slices as \(t\nearrow t_0\). Since a singular Ricci flow may have several connected components at a fixed time, the statement is naturally formulated componentwise.

\begin{thm}[Gromov--Hausdorff convergence of singular Ricci flow time-slices]
Let \((\mathcal M,\mathfrak t,\partial_{\mathfrak t},g)\) be the singular Ricci flow with a normalized, closed, orientable three-dimensional initial condition. For each $t_0 >1$, let \(Z_{t_0}'\) be a connected component of \(Z_{t_0}\), and let \(\mathcal R'\subset \mathcal R_{[1,t_0)}\) be the branch determined by $Z_{t_0}'$; see Definition~\ref{defn:componentbranch}. Then
\[
    (\overline{\mathcal R_t'},d_{g_t})
    \xrightarrow[t\nearrow t_0]{\GHconvtext}
    (Z_{t_0}',d_{t_0}^Z),
\]
where \(\overline{\mathcal R_t'}\) denotes the metric completion of \((\mathcal R_t',d_{g_t})\). In particular, each connected component of \((Z_{t_0},d_{t_0}^Z)\) is a compact geodesic metric space.
\end{thm}
Note that if \(Z_{t_0}\) contains no singular points, as is the case for $t_0$ outside a subset of $\mathscr{H}^{1/2}$-measure zero by Theorem~\ref{thm:intro_singular_structure}(2), then
\[
    Z_{t_0}=\mathcal R_{t_0}
    \qquad\text{and}\qquad
    d^Z_{t_0}=d_{g_{t_0}};
\]
see Proposition~\ref{prop:distance}. Thus, the theorem above can be interpreted as the past-continuity, in the Gromov--Hausdorff sense, of the regular time-slices $(\mathcal R_t,d_{g_t})$ as \(t\nearrow t_0\). More generally, if one considers \(Z_t\) with its intrinsic distance \(d^Z_t\), then the same theorem implies the past-continuity of $(Z_t,d^Z_t)$ as \(t\nearrow t_0\); see Proposition~\ref{prop:pastcont}.

The next main theme is the structure of the singular set. In the general theory of noncollapsed Ricci flow limit spaces developed in \cite{fangli2025structure}, the singular set admits a stratification according to the symmetry of tangent flows. In dimension three, the situation is especially rigid: every singular tangent flow is quotient cylindrical. Thus, the singular set behaves like a one-dimensional parabolic object. By generalizing the cylindrical strong uniqueness and parabolic rectifiability results of \cite{fangli2025loja,fangli2025recti}, and using the heat kernel estimates developed here, we prove the following structure theorem.

\begin{thm}[Structure of the singular set]
\label{thm:intro_singular_structure}
The singular set \(\mathcal S\subset Z\) satisfies the following properties.
\begin{enumerate}[label=\textnormal{(\arabic*)}]
\item \(\mathcal S\) is horizontally parabolic \(1\)-rectifiable.

\item The \(1/2\)-dimensional Hausdorff measure of the set of singular times vanishes:
\[
    \mathscr H^{1/2}(\mathfrak t(\mathcal S))=0.
\]

\item For any $T>1$ and \(z_0\in Z_{[1, T]}\) with \(\mathfrak t(z_0)-2r_0^2>1\), we have
\[
    \mathscr H^1\bigl(\mathcal S\cap B^*(z_0,r_0)\bigr)\le C(T)r^5_0.
\]

\item For any $t \ge 1$, the Minkowski dimension of \(\mathcal S_t\subset Z_t\) satisfies
\[
    \dim_{\mathscr M}(\mathcal S_t)\le 1
\]
with respect to the intrinsic metric \(d_t^Z\).
\end{enumerate}
\end{thm}

Note that Theorem~\ref{thm:intro_singular_structure}(2) improves \cite[Theorem 1.4]{kleinerlott2020singular2}. We further conjecture that \(\mathfrak t(\mathcal S)\) should be a finite set.

We also prove global curvature and diameter estimates for singular Ricci flows. Let \(D_{g_t}\) denote the sum of the diameters of all connected components of \(\mathcal R_t\) with respect to \(d_{g_t}\). Then we have the following theorem.

\begin{thm}[\(L^1\)-curvature and diameter bounds]
In the above setting, for every $T>2$ and \(t\in[2,T]\),
\[
    \int_{\mathcal R_t} |\Rm|\,\d V_{g_t}
    \le
    \int_{\mathcal R_t} r_{\Rm}^{-2}\,\d V_{g_t}
    \le C,
\]
and
\[
    D_{g_t}\le C,
\]
where \(C\) depends only on \(T\) and the diameter of the initial metric.
\end{thm}

These estimates may be viewed as global versions, for singular Ricci flows, of the \(L^1\)-curvature bound and of a generalized form of Perelman's bounded diameter theorem.

Finally, we explain how the same ideas extend to certain higher-dimensional Ricci flows. Under the positive isotropic curvature condition, in dimensions for which the surgery theory and the classification of \(\kappa\)-solutions are available, the only noncompact singular model is the standard cylinder \(S^{n-1}\times\mathbb R\). In this setting, the same arguments yield analogous convergence and structure results for the spacetime completion of the corresponding singular Ricci flow.

The paper is organized as follows. In Section 2, we recall the relevant background on Ricci flows with surgery, singular Ricci flows, and heat kernels. In Section 3, we prove the Gromov--Hausdorff convergence theorem at the first singular time of a closed three-dimensional Ricci flow. In Section 4, we establish the heat kernel estimates on singular Ricci flows that are needed later. In Section 5, we construct the spacetime completion of a singular Ricci flow, define the intrinsic metrics on its time-slices, prove the curvature and diameter estimates, establish the Gromov--Hausdorff convergence of time-slices, and derive the structure theorems for the singular set. In Section 6, we discuss extensions to higher-dimensional Ricci flows with positive isotropic curvature.

\bigskip

\noindent\textbf{Acknowledgments}: The author thanks Hanbing Fang for proofreading the first draft. The author is supported by the National Key Research and Development Program of China 2025YFA1018200, NSFC-12522105, YSBR-001, and research funds from the University of Science and Technology of China and the Chinese Academy of Sciences.

\section{Preliminaries}

In this section, we recall some basic definitions and results concerning three-dimensional Ricci flows with surgery and singular Ricci flows from \cite{perelman2002entropy,perelman2003surgery,kleinerlott2017singular,bamlerkleiner2022,morgantian2007,caozhu2006}. For simplicity, except where explicitly stated otherwise, throughout this paper we assume that all Riemannian manifolds under consideration are \textbf{orientable}.

\subsection{Ricci flows with surgery}

\begin{defn}[Normalized Riemannian manifold] \label{def:normal}
We say that a compact $n$-dimensional Riemannian manifold $(M^n,g)$ is \textbf{normalized} if
\begin{align*}
    |\Rm|+\mathrm{inj}^{-1}\le 10^{-n},
\end{align*}
where $\mathrm{inj}$ denotes the injectivity radius of $(M,g)$.
\end{defn}

In this paper, we mostly consider Ricci flows whose initial data are normalized Riemannian manifolds.

\begin{defn}[$\kappa$-noncollapsed]
A Ricci flow $(M^3,(g_t)_{t\in I})$ is called \textbf{$\kappa$-noncollapsed} at scales less than $r_0$, where $\kappa$ and $r_0$ are positive constants, if for any $(x,t)\in M\times I$ and any $r\in(0,r_0]$, the condition
\begin{align*}
    |\Rm|\le r^{-2}\quad \text{on } B_{g_t}(x,r)
\end{align*}
implies that
\begin{equation*}
    |B_{g_t}(x,r)|_t \ge \kappa r^3.
\end{equation*}
\end{defn}

\begin{defn}[$\kappa$-solution]
An ancient complete Ricci flow $(M^3,g_t)_{t\in(-\infty,0]}$ is called a \textbf{$\kappa$-solution} if $\Rm\ge 0$ and is uniformly bounded on $M\times(-\infty,0]$, and if $(M^3,g_t)$ is $\kappa$-noncollapsed at all scales for every $t\in(-\infty,0]$.
\end{defn}

The classification of all three-dimensional $\kappa$-solutions is known; see \cite{hamilton1993formations,perelman2002entropy,naber2010noncompact,niwallach2008classification,caochenzhu2008,brendle2020dim3,brendledaskalopoulossesum2021}. The complete list consists of the following (up to quotients where allowed):
\begin{enumerate}[label=\textnormal{(\roman*)}]
    \item Ricci shrinkers. More precisely, these are the standard Ricci flows on $\R^3$, $S^3$, or $S^2\times \R$;
    \item the Bryant steady soliton;
    \item Perelman's ancient oval on $S^3$.
\end{enumerate}

\begin{defn}[$\ep$-closeness]
Let $(M_i,g_i,x_i)$, for $i=1,2$, be two pointed Riemannian manifolds. We say that $(M_1,g_1,x_1)$ is \textbf{$\ep$-close} to $(M_2,g_2,x_2)$ at scale $r$ if there exists a map
\[
\varphi:B_{g_2}(x_2,\ep^{-1})\to M_1
\]
such that $\varphi(x_2)=x_1$ and
\begin{equation*}
\bigl\| r^{-2}\varphi^*g_1-g_2 \bigr\|_{C^{[\ep^{-1}]}\left(B_{g_2}(x_2,\ep^{-1})\right)}<\ep,
\end{equation*}
where the $C^{[\ep^{-1}]}$-norm is taken with respect to $g_2$. 
\end{defn}

Next, we recall the following definition from \cite{kleinerlott2017singular}.

\begin{defn}[Ricci flow spacetime]\label{def_RF_spacetime}
A Ricci flow spacetime over an interval $I\subset \R$ is a tuple $(\mathcal M,\mathfrak t,\partial_{\mathfrak t},g)$ with the following properties:
\begin{enumerate}[label=\textnormal{(\arabic*)}]
\item $\mathcal M$ is a four-dimensional smooth manifold with smooth boundary $\partial\mathcal M$, and $\partial\mathcal M$ is a disjoint union of smooth three-dimensional manifolds.

\item $\mathfrak t:\mathcal M\to I$ is a smooth function without critical points. For any $t\in I$, we denote by
\[
\mathcal M_t:=\mathfrak t^{-1}(t)\subset\mathcal M
\]
the time-$t$ slice of $\mathcal M$.

\item $\mathfrak t(\partial\mathcal M)\subset \partial I$.

\item $\partial_{\mathfrak t}$ is a smooth vector field on $\mathcal M$ satisfying
\[
\partial_{\mathfrak t}\mathfrak t \equiv 1.
\]

\item $g$ is a smooth inner product on the spatial subbundle $\ker(d\mathfrak t)\subset T\mathcal M$. For any $t\in I$, we denote by $g_t$ the restriction of $g$ to the time-$t$ slice $\mathcal M_t$.

\item $g$ satisfies the Ricci flow equation
\[
\mathcal L_{\partial_{\mathfrak t}}g=-2\Ric(g).
\]
Here $\Ric(g)$ denotes the symmetric $(0,2)$-tensor on $\ker(d\mathfrak t)$ that restricts to the Ricci tensor of $(\mathcal M_t,g_t)$ for all $t\in I$.
\end{enumerate}
\end{defn}

Next, we recall the definition of Ricci flow with surgery from \cite{perelman2003surgery}, following the variant in \cite{Bessieres2011}, in which the surgeries are done before singular times.

\begin{defn}[Ricci flow with surgery]\label{def_RF_surgery}
A Ricci flow with surgery $\mathcal M$ consists of the following data:
\begin{enumerate}[label=\textnormal{(\arabic*)}]
\item A collection of Ricci flows
\[
\{(M_k,g_k(t))_{t\in [t_k^-,t_k^+]}\}_{1\le k\le N},
\]
where $N\le \infty$, each $M_k$ is a compact (possibly empty) manifold, $t_k^+=t_{k+1}^-$ for all $1\le k<N$.

\item A collection of isometric embeddings
\[
\{\psi_k:X_k^+\to X_{k+1}^-\}_{1\le k<N},
\]
where $X_k^+\subset M_k$ and $X_{k+1}^-\subset M_{k+1}$ are compact three-dimensional submanifolds with boundary. The sets $X_k^\pm$ are the regions that survive the transition from one flow to the next, and the maps $\psi_k$ identify these surviving regions.
\end{enumerate}

We say that $t$ is a \textbf{surgery time} if $t=t_k^+=t_{k+1}^-$ for some $1\le k<N$. If $t=t_k^+$ for some $1\le k\le N$, we set
\[
(\mathcal M_t^-,g_t)=(M_k,g_k(t_k^+)).
\]
If, in addition, $t\neq t_N^+$, then we set
\[
(\mathcal M_t^+,g_t)=\bigl(M_{k+1},g_{k+1}(t_{k+1}^-)\bigr).
\]
If $t\in (t_k^-,t_k^+)$, then we set
\[
(\mathcal M_t,g_t)=(\mathcal M_t^\pm,g_t)=(M_k,g_k(t)).
\]
\end{defn}

Given a Ricci flow with surgery $\mathcal M$ as above, one can define an associated Ricci flow spacetime $\mathcal M'$ as follows. For each $1\le k<N$, let $\mathrm{Int}(X_k^+)$ and $\mathrm{Int}(X_{k+1}^-)$ denote the interiors of $X_k^+$ and $X_{k+1}^-$, respectively. We then construct $\mathcal M'$ by gluing
\begin{equation*}
\bigl(M_k\times [t_k^-,t_k^+)\bigr)\cup \bigl(\mathrm{Int}(X_k^+)\times \{t_k^+\}\bigr), \qquad 1\le k\le N,
\end{equation*}
using the maps $\psi_k$, and removing the subsets
\begin{equation*}
\bigl(M_{k+1}\setminus \mathrm{Int}(X_{k+1}^-)\bigr)\times \{t_{k+1}^-\}, \qquad 1\le k<N.
\end{equation*}
The time functions, time vector fields, and metrics descend to $\mathcal M'$, yielding a tuple $(\mathcal M',\mathfrak t,\partial_{\mathfrak t},g)$ defined over $[t_1^-,t_N^+]$.

For any parameter $\ep>0$, there exist nonincreasing functions
\[
\delta,r,\kappa:[0,\infty)\to [0,\infty)
\]
such that, for any normalized Riemannian manifold, a Ricci flow with surgery starting from this initial condition is defined for all time with respect to $\ep$, $\delta$, $r$, and $\kappa$; see \cite[Proposition 5.1]{perelman2003finite} and \cite[Proposition 77.2]{kleinerlott2008notes}. More precisely, we have the following theorem.

\begin{thm}\label{thm:ricciflowsurgery}
For any normalized Riemannian manifold $(M^3,g_0)$, there exists a Ricci flow with surgery $\mathcal M$ defined on $[0,\infty)$ starting from $(M^3,g_0)$ such that the following properties hold.
\begin{enumerate}[label=\textnormal{(\alph*)}]
\item $\mathcal M$ is a Ricci flow with $(r(t),\delta(t))$-cutoff at every surgery time $t$, in the sense of \ucite[Definition 73.1]{kleinerlott2008notes}.

\item $\mathcal M$ satisfies the $(r,\ep)$-canonical neighborhood assumption. That is, for any $t\ge 0$ and any $x\in \mathcal M_t^\pm$, if $\scal(x,t)\ge r(t)^{-2}$, then $(\mathcal M_t^\pm,g_t,x)$ is $\ep$-close to a $\kappa$-solution; cf. \ucite[Definition 69.1]{kleinerlott2008notes}.

\item For any $t\ge 0$, $(\mathcal M_t^\pm,g_t)$ is $\kappa(t)$-noncollapsed on scales smaller than $\ep$.

\item $\mathcal M$ satisfies the Hamilton--Ivey pinching condition.
\end{enumerate}
\end{thm}

\subsection{Singular Ricci flows and heat kernels} 

First, we recall the construction of singular Ricci flows from \cite{kleinerlott2017singular}.

Given a normalized Riemannian manifold $(M^3,g_0)$, let $\delta_i=\delta_i(t)\to 0$ be a sequence of surgery parameter functions. By Theorem~\ref{thm:ricciflowsurgery}, for each $i$ there exists a Ricci flow with surgery $\mathcal M_i$ with initial condition $(M^3,g_0)$. Let $\mathcal M_i'$ denote the associated Ricci flow spacetime. Using a compactness argument, Kleiner and Lott proved in
\cite{kleinerlott2017singular} that, after passing to a subsequence, one has
\[
\mathcal M_i' \to \mathcal M.
\]
The limiting Ricci flow spacetime $(\mathcal M,\mathfrak t,\partial_{\mathfrak t},g)$ is called the \textbf{singular Ricci flow}, and it satisfies the following properties:
\begin{enumerate}[label=\textnormal{(\alph*)}]
\item for any $T\ge 0$, the scalar curvature function $\scal:\mathcal M_{\le T}\to \R$ is bounded below and proper;

\item $\mathcal M$ satisfies the Hamilton--Ivey pinching condition;

\item for some parameter $\ep>0$ and some functions $\kappa,r:[0,\infty)\to [0,\infty)$, $\mathcal M$ is $\kappa$-noncollapsed below scale $\ep$ and satisfies the $(r,\ep)$-canonical neighborhood assumption;

\item $\mathcal M$ is $0$-complete in the sense of \cite[Definition 5.3]{bamlerkleiner2022}.
\end{enumerate}

In \cite[Theorem 1.1]{bamlerkleiner2022}, it was proved that if $\ep<\ep_{\mathrm{can}}$, where $\ep_{\mathrm{can}}$ is a universal constant, then any singular Ricci flow satisfying the above conditions with the same initial data is unique. In the following, we will always assume $\ep<\ep_{\mathrm{can}}$, and refer to the Ricci flow spacetime $\mathcal M$ as the singular Ricci flow starting from $(M^3,g_0)$.

%For later use, we have the following definition.

%\begin{defn} \label{def:regulartime}
%For the singular Ricci flow $\MM$, a time $t$ is called \textbf{regular} if $\MM_t$ is compact. Otherwise, $t$ is called %\textbf{singular}.    
%\end{defn}

For a singular Ricci flow $\mathcal M$, there exists a heat kernel $K(x;\cdot)$, constructed in \cite[Section 5]{lai2021}. We briefly recall this construction. Fix $x_0\in \mathcal M$. For the Ricci flows with surgery $\mathcal M_i$ and the associated Ricci flow spacetimes $\mathcal M_i'$ above, let $x_{i0}\in \mathcal M_i'$ be a sequence converging to $x_0$.

Suppose that $\mathcal M_i$ is given by the data
\[
\{(M_k,g_k(t))_{t\in [t_k^-,t_k^+]}\}_{1\le k\le N}
\]
as in Definition~\ref{def_RF_surgery}, where $t_N^+=t_{i0}:=\mathfrak t(x_{i0})$. On $M_N\times [t_N^-,t_N^+]$, let $u_i$ be the conjugate heat kernel of $(M_N,g_N(t))$ based at $(x_{i0},t_{i0})$. We extend $u_i$ to $M_{N-1}\times \{t_{N-1}^+\}$ by setting
\[
u_i(x,t_{N-1}^+)=u_i(x,t_N^-)
\quad \text{if } x\in \mathrm{Int}(X_{N-1}^+),
\]
and
\[
u_i(x,t_{N-1}^+)=0
\quad \text{if } x\in M_{N-1}\setminus \mathrm{Int}(X_{N-1}^+).
\]
Then, for any $(x,t)\in M_{N-1}\times [t_{N-1}^-,t_{N-1}^+]$, we define $u_i$ by the reproduction formula
\begin{equation*}
u_i(x,t)=\int_{\mathrm{Int}(X_{N-1}^+)} u_i(y,t_{N-1}^+)\,K_{N-1}(y,t_{N-1}^+;x,t)\,\d V_{g(t^+_{N-1})}(y),
\end{equation*}
where $K_{N-1}(y,t_{N-1}^+;x,t)$ denotes the heat kernel on $M_{N-1}\times [t_{N-1}^-,t_{N-1}^+]$. Iterating this procedure defines $u_i$ on all of $\mathcal M_i$. Thus, we set $K_i(x_0;\cdot)=u_i$. As $i\to\infty$, the functions $K_i(x_0;\cdot)$ converge to the conjugate heat kernel $K(x_0;\cdot)$ on $\mathcal M$.

The following result was proved in \cite[Section 5]{lai2021}.

\begin{thm}\label{thm:heatkernel}
The heat kernel $K$ constructed above on the singular Ricci flow $\mathcal M$ satisfies the following properties:
\begin{enumerate}[label=\textnormal{(\alph*)}]
\item For any $x_0\in \mathcal M$, the function
\[
K(x_0;\cdot):\mathcal M_{t<\mathfrak t(x_0)}\to \R
\]
is a smooth solution to the conjugate heat equation
\[
\square^*K(x_0;\cdot)=(-\partial_{\mathfrak t}-\Delta+\scal)K(x_0;\cdot)=0.
\]

\item For any $h\in C_c^0(\mathcal M)$ and any $x_0\in \mathcal M$,
\begin{equation*}
\lim_{t\nearrow \mathfrak t(x_0)}\int_{\mathcal M_t} K(x_0;x)\,h(x)\,\d V_{g_t}(x)=h(x_0).
\end{equation*}

\item Suppose that for some $r_0,T>0$, we have $|\Rm|\le r_0^{-2}$ on
\[
P_0:=P(x_0,r_0,-r_0^{2})
\]
(see \ucite[Definition 2.4]{lai2021}) and that $\mathfrak t(x_0)<T$. Then there exists a constant $C=C(r_0,T)>0$ such that
\[
K(x_0;\cdot)\le C
\quad \text{on } \mathcal M_{t<\mathfrak t(x_0)}\setminus P_0.
\]

\item For any $x_0\in \mathcal M$ and any $t<\mathfrak t(x_0)$, we have
\begin{equation*}
\int_{\mathcal M_t} K(x_0;x)\,\d V_{g_t}(x)=1.
\end{equation*}

\item For any $x,y\in \mathcal M$ with $\mathfrak t(x)>\mathfrak t(y)$, we have
\begin{equation*}
K(x;y)=\int_{\mathcal M_t} K(x;z)\,K(z;y)\,\d V_{g_t}(z)
\end{equation*}
for every $t\in (\mathfrak t(y),\mathfrak t(x))$.

\item For any $x_0\in \mathcal M$, the function
\[
K(\cdot;x_0):\mathcal M_{t>\mathfrak t(x_0)}\to \R
\]
is a smooth solution to the heat equation
\[
\square K(\cdot;x_0)=(\partial_{\mathfrak t}-\Delta)K(\cdot;x_0)=0.
\]

\item For any $h\in C_c^0(\mathcal M)$ and any $x_0\in \mathcal M$,
\begin{equation*}
\lim_{t\searrow \mathfrak t(x_0)}\int_{\mathcal M_t} K(x;x_0)\,h(x)\,\d V_{g_t}(x)=h(x_0).
\end{equation*}

\item Suppose that for some $r_0,T>0$, we have $|\Rm|\le r_0^{-2}$ on
\[
P_0:=P(x_0,r_0,r_0^{2})
\]
(see \ucite[Definition 2.4]{lai2021}) and that $\mathfrak t(x_0)<T$. Then there exists a constant $C=C(r_0,T)>0$ such that
\[
K(\cdot;x_0)\le C
\quad \text{on } \mathcal M_{t>\mathfrak t(x_0)}\setminus P_0.
\]
\end{enumerate}
\end{thm}

We will also need the following refined estimate from \cite[Theorem 5.6]{lai2021}, which improves Theorem~\ref{thm:heatkernel}(c) at points of high curvature.

\begin{thm}\label{thm:heatkernel2}
Given $r_0,T>0$ with $r_0^2<T$, and $m\in \N$, there exists a constant $C_m=C(r_0,T,m)>0$ such that the following holds for any $t_0\in (r_0^2,T)$.

Let $\mathcal M$ be a singular Ricci flow with a normalized initial condition. For any $x_0\in \mathcal M$ with $\mathfrak t(x_0)=t_0$, suppose that
\[
|\Rm|\le r_0^{-2}\quad \text{on } P_0:=P(x_0,r_0,-r_0^{2}).
\]
Then, for any $x\in \mathcal M_{t<\mathfrak t(x_0)}\setminus P_0$, we have
\begin{equation*}
K(x_0;x)\,\scal^m(x)\le C_m.
\end{equation*}
\end{thm}

We end this section with the following definition.

\begin{defn}[Curvature radius]\label{def:curvatureradius}
For any $x \in \MM$, the curvature radius $r_{\Rm}(x)$ is defined to be the supremum of all $r>0$ such that $B_{g_t}(x,r)$ is relatively compact in $\MM_t$, and the product domain 
	\begin{align*}
B_{g_t}(x,r)\times \lc [\t(x)-r^2, \t(x)+r^2] \cap [0, \infty) \rc
	\end{align*}
is unscathed and satisfies the curvature bound $|\Rm| \le r^{-2}$.
\end{defn}

\section{Gromov--Hausdorff convergence at the first singular time}\label{sec:first}

In this section, we assume that $(M^3,(g_t)_{t\in [0,T)})$ is a closed Ricci flow with an orientable, normalized initial condition~$(M,g_0)$. We assume that $T<\infty$ is the first singular time. Note that by the curvature bound in Definition~\ref{def:normal}, we have $T \ge 2$; see, for instance, \cite[Equation (10.14)]{chowetal2010part1-4}.

Moreover, by the monotonicity of the Nash entropy, we have
\begin{align} \label{eq:nashlower}
\inf_{\tau \in (0, t]}\mathcal N_{(x,t)}(\tau)
\ge
\mathcal N_{(x,t)}(t)
\ge
\inf_{\tau\in (0,T]} \boldsymbol{\mu}(g_0,\tau)
\ge -Y_0,
\end{align}
for any $(x,t)\in M \times [0, T)$. Here, $Y_0$ is a constant depending only on $T$ by Definition~\ref{def:normal}. Thus, the Ricci flow has entropy bounded below by $-Y_0$ in the sense of \cite[Definition 2.20]{fangli2025structure}.

\subsection{Sequential compactness of time-slices}

We first recall the following $L^1$-curvature bound from \cite[Theorem 1.7]{fangli2025recti} and the resolution of Perelman's bounded diameter conjecture from \cite[Theorem 1.8]{fangli2025recti}.

\begin{thm}\label{thm:l1bound}
There exists a constant $C$, depending only on the flow, such that for any $t\in [0, T)$,
\begin{align*}
	\int_M |\Rm|\,\d V_{g_t}+\mathrm{diam}_{g_t}(M)\le C.
\end{align*}
\end{thm}

Next, we prove the following proposition.

\begin{prop}\label{prop:ballpacking}
For any $t\in [0, T)$ and any $\ep>0$, the maximal cardinality of an $\ep$-separated subset of $(M,g_t)$ is bounded by a constant $C_\ep$ independent of $t$.
\end{prop}

\begin{proof}

We fix a time $t\in [0, T)$. Throughout the proof, all constants $C_i$ and $c_i$ may depend on the given Ricci flow, but are independent of $t$.

Define a measure $\mu_t$ by
\begin{align*}
\mu_t(A):=\int_A (|\scal|+1)\,\d V_{g_t}
\end{align*}
for any Borel subset $A\subset M$. By Theorem~\ref{thm:l1bound} and \cite[Proposition 9.3]{fangli2025structure}, there exists a constant $C_1>0$ such that
\begin{align}\label{eq:3001}
\mu_t(M)\le C_1.
\end{align}

Next, we claim that for any $x\in M$ and any sufficiently small $\ep>0$, if $B_{g_t}(x,\ep)\neq M$, then
\begin{align}\label{eq:001}
\mu_t(B_{g_t}(x,\ep))\ge c_\ep>0,
\end{align}
where $c_\ep$ is a constant depending on the Ricci flow and on $\ep$, but independent of $t$. Once \eqref{eq:001} is established, the conclusion follows immediately from \eqref{eq:3001}.

By the canonical neighborhood theorem (see \cite[Theorem 12.1]{perelman2002entropy} and \cite[Theorem 52.7]{kleinerlott2008notes}), for any $\delta>0$, there exists $r_\delta>0$ such that if $\scal(x,t)\ge r_\delta^{-2}$, then $(M,g_t,x)$ is $\delta$-close to a $\kappa$-solution at scale $\scal^{-\frac{1}{2}}(x, t)$. We will choose $\delta$ as a small universal constant later. First, we require that $r_\delta$ is sufficiently small so that $r_\delta \le \delta \ep$.

To prove \eqref{eq:001}, we consider two cases.

\textbf{Case 1.} There exists a point $y\in B_{g_t}(x,\ep/2)$ such that $\scal(y,t)<r_\delta^{-2}$.

In this case, it follows from \cite[Lemma 3.1]{kleinerlott2017singular} that
\[
\scal \le \eta^{-2}r_\delta^{-2}\qquad \text{on } B_{g_t}(y,\eta r_\delta)
\]
for some universal constant $\eta\in (0,1)$. Since $\eta r_{\delta} \le \ep$, Perelman's no-local-collapsing theorem \cite{perelman2002entropy} implies that
\begin{align*}
|B_{g_t}(x,\ep)|_t \ge |B_{g_t}(y,\eta r_\delta)|_t \ge c_1 r_\delta^3>0.
\end{align*}

Therefore, in this case we obtain
\begin{align}\label{eq:002}
\mu_t(B_{g_t}(x,\ep))\ge |B_{g_t}(x,\ep)|_t \ge c_1 r_\delta^3>0.
\end{align}

\textbf{Case 2.} $\scal \ge r_\delta^{-2}$ on $B_{g_t}(x,\ep/2)$.

In this case, fix a point $y\in \partial B_{g_t}(x,\ep/2)$ and choose a minimizing geodesic $\gamma(s)$ from $x$ to $y$ with respect to $g_t$. By assumption, $(M,g_t,\gamma(s))$ is $\delta$-close to a $\kappa$-solution at scale $\scal^{-\frac{1}{2}}(\gamma(s), t)$ for every $s\in [0,\ep/2]$. Next, choose disjoint balls $\{B_{g_t}(x_i,r_i)\}_{1\le i\le N}$, where $x_i=\gamma(s_i)$ and $r_i=\scal^{-1/2}(x_i,t)$, such that the balls $B_{g_t}(x_i,5r_i)$ cover $\gamma$.

For each $i$, the pointed manifold $(M,g_t,x_i)$ is $\delta$-close at scale $r_i$ to one of the following: (a) a $\delta$-neck; (b) a $\delta$-cap; or (c) a closed manifold of constant positive sectional curvature; see \cite[A.8]{kleinerlott2017singular}. In all three cases, comparison with the corresponding model implies that
\begin{align}\label{eq:003}
\int_{B_{g_t}(x_i,r_i)} |\scal|\,\d V_{g_t} \ge c_2 r_i>0.
\end{align}

Since $r_i\le r_\delta \le \delta \ep$, the pointed manifold $(M,g_t,x_i)$ cannot be $\delta$-close at scale $r_i$ to a closed manifold of constant positive sectional curvature. Indeed, otherwise the diameter of $(M,g_t)$ would be bounded by $C_2 r_i \le \ep$, where $C_2$ is a universal constant, contradicting the assumption that $B_{g_t}(x,\ep)\neq M$. Moreover, since $\gamma$ is a minimizing geodesic and all $x_i$ lie on $\gamma$, the number of indices $i$ for which $(M,g_t,x_i)$ is $\delta$-close to a $\delta$-cap is at most $2C_3$, where $C_3$ is a universal constant.

\begin{figure}[tbp]
    \centering
    \includegraphics[width=0.94\linewidth,trim=0 170 0 235,clip]{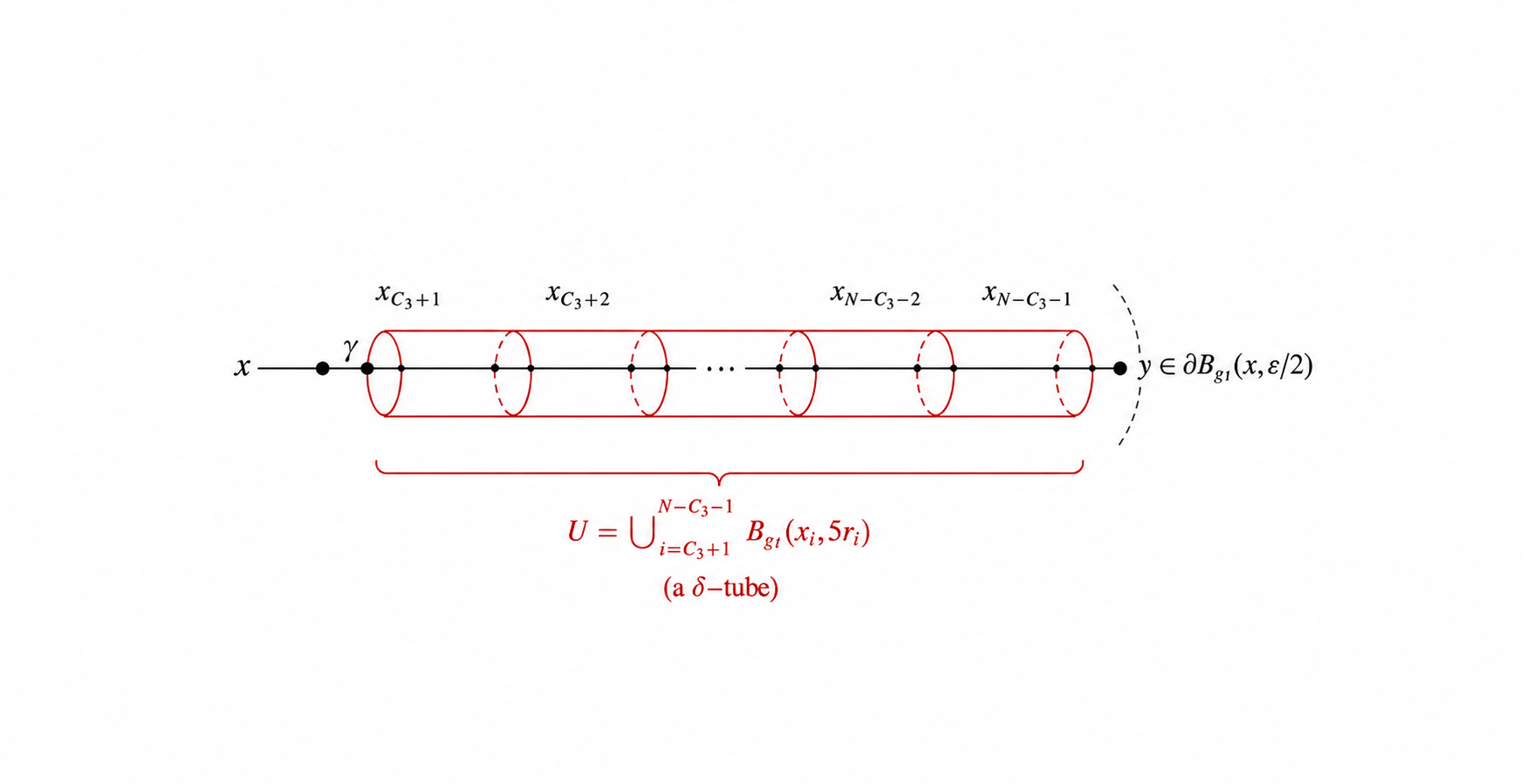}
    \caption{A $\delta$-neck.}
    \label{fig:delta-neck}
\end{figure}

Without loss of generality, we may assume that for every $C_3+1\le i\le N-C_3-1$, the pointed manifold $(M,g_t,x_i)$ is $\delta$-close to a $\delta$-neck at scale $r_i$; see Figure~\ref{fig:delta-neck}. Hence the open set
\begin{align*}
U:=\bigcup_{C_3+1\le i\le N-C_3-1} B_{g_t}(x_i,5r_i)
\end{align*}
forms a $\delta$-tube; see \cite[A.3]{kleinerlott2017singular}. The total length of $\gamma\cap U$ satisfies
\begin{align*}
\mathrm{length}(\gamma\cap U)\ge \ep/2-20C_3 r_\delta \ge \frac{\ep}{4},
\end{align*}
where the last inequality holds because $r_\delta \le \delta \ep$.

On the other hand, each segment $\gamma\cap B_{g_t}(x_i,5r_i)$, for $C_3+1\le i\le N-C_3-1$, is almost parallel to the axis of the model cylinder. Therefore,
\begin{align*}
10\sum_{i=C_3+1}^{N-C_3-1} r_i \ge \frac{1}{2}\mathrm{length}(\gamma\cap U)\ge \frac{\ep}{8}.
\end{align*}
Combining this with \eqref{eq:003}, we obtain
\begin{align}\label{eq:005}
\mu_t(B_{g_t}(x, \ep)) \ge \int_{B_{g_t}(x, \ep)} |\scal|\,\d V_{g_t} \ge \sum_i \int_{B_{g_t}(x_i, r_i)} |\scal|\,\d V_{g_t} \ge c_2 \sum_i r_i \ge c_3 \ep>0.
\end{align}

Combining \eqref{eq:002} and \eqref{eq:005}, we conclude that \eqref{eq:001} holds. This completes the proof.
\end{proof}

A direct consequence of Proposition~\ref{prop:ballpacking}, together with the uniform diameter bound in Theorem~\ref{thm:l1bound} and Gromov's compactness theorem \cite{gromov1981}, is the following existence result for sequential limits as~$t\nearrow T$.

\begin{cor}
In the above setting, for any sequence $\{t_j\}$ with $t_j\nearrow T$, after passing to a subsequence if necessary, we have
\begin{align} \label{eq:grsequence}
   (M, d_{g_{t_j}}) \xrightarrow[t_j \nearrow T]{\GHconvtext} (X_{\GH}, d_{\GH}), 
\end{align}  
where $d_{g_t}$ denotes the distance induced by $g_t$. The limit $(X_{\GH},d_{\GH})$ is a compact geodesic metric space.
\end{cor}

Note that the limit metric space $(X_{\GH},d_{\GH})$ may \emph{a priori} depend on the choice of the sequence $\{t_j\}$ and hence need not be unique. In what follows, we will prove that the limit metric space is in fact unique.

\subsection{The spacetime completion and the terminal slice}

We consider the $d^*$-distance on $M\times [1,T)$, defined as in \cite[Definition 3.2]{fangli2025structure}. More precisely, for any $x^*=(x,t), y^*=(y,s)\in M\times [1,T)$ with $s\le t$, we define
\begin{align}\label{defndstar}
 d^*(x^*,y^*):=\inf_{1\le \tau\le s}
 \max\left\{\sqrt{t-\tau},\,
 d_{W_1}^{g_{\tau}}\bigl(\nu_{x^*;\tau},\nu_{y^*;\tau}\bigr)\right\}.
\end{align}
Here $\nu$ denotes the conjugate heat kernel measure of the Ricci flow; see \cite[Definition 2.6]{fangli2025structure}. Moreover, $d_{W_1}^{g_t}$ denotes the $W_1$-Wasserstein distance with respect to $(M,d_{g_t})$; see \cite[Definition 2.1]{fangli2025structure}. Taking $\tau=1$ in \eqref{defndstar} shows that $M\times[1,T)$ has bounded diameter with respect to $d^*$ (see \cite[(9.1)]{fangli2025structure}). By \cite[Proposition 3.9(2)]{fangli2025structure}, the time function $\mathfrak t$ is $2$-H\"older continuous with respect to $d^*$, namely,
\begin{equation*}
\bigl|\mathfrak t(x^*)-\mathfrak t(y^*)\bigr|\le d^*(x^*,y^*)^2.
\end{equation*}

We then define $(Z,d_Z,\mathfrak t)$ to be the metric completion of $M\times [1,T)$ with respect to $d^*$. By construction, the completion adds only points in
\[
Z_T:=\mathfrak t^{-1}(T)\subset Z.
\]
Moreover, it is proved in \cite[Section 9]{fangli2025structure} that $(Z,d_Z,\mathfrak t)$ is a noncollapsed Ricci flow limit space over $[1,T]$ in the sense of \cite[Definition 3.21]{fangli2025structure}.

As in \cite[Definition 6.1]{fangli2025structure}, for each $t\in [1,T)$, we define a distance function on $Z_t:=\mathfrak t^{-1}(t)\subset Z$ by
\begin{align}\label{defnd0}
d_t^Z(x,y):=\lim_{s\nearrow t} d_{W_1}^{g_s}(\nu_{x;s},\nu_{y;s}) \in [0,+\infty]
\end{align}
for any $x,y\in Z_t$, where the limit exists by \cite[Lemma 5.29]{fangli2025structure}. If $t\in [1,T)$, then $(Z_t,d_t^Z)$ is precisely $(M,d_{g_t})$. Moreover, $(Z_T,d_T^Z)$ is isometric to Bamler's asymptotic boundary introduced in \cite[Section 2.6]{bamler2020structure}; see \cite[Lemma 9.1]{fangli2025structure}.

For the remainder of this section, we fix a sequence $\{t_j\}$ with $t_j\nearrow T$ such that \eqref{eq:grsequence} holds.

\begin{prop}\label{prop:embh}
There exists an isometric embedding
\begin{align*}
\iota:(Z_T,d_T^Z)\longrightarrow (X_{\GH},d_{\GH}).
\end{align*}
\end{prop}

\begin{proof}
We first show that $(Z_T,d_T^Z)$ is separable, i.e.,\ it contains a countable dense subset. Fix $\ep>0$, and choose a countable $\ep$-separated set $\{z_i\}_{i=1}^N\subset Z_T$, where $N\in [1,\infty]$. Note that a maximal $\ep$-separated set in $Z_T$ may be uncountable.

For $z_1$, let $x_{1,j}^*$ be an $H_3$-center of $z_1$ on $M\times \{t_j\}$, where $H_3:=\pi^2+4$; see \cite[Definition 6.12]{fangli2025structure}. After passing to a subsequence of $\{t_j\}$, we may assume that $x_{1,j}^*$ converges to some point $x_1\in X_{\GH}$ in the Gromov--Hausdorff sense.

Repeating this construction for $z_2,z_3,\dots$, and then passing to a diagonal subsequence of $\{t_j\}$, still denoted by $\{t_j\}$, we may assume that for each $z_i$, the corresponding $H_3$-centers $x_{i,j}^*\in M\times \{t_j\}$ converge to some point $x_i\in X_{\GH}$ in the Gromov--Hausdorff sense.

We now define a map
\[
\iota:\{z_i\}_{i=1}^N\to \{x_i\}_{i=1}^N
\]
by setting
\[
\iota(z_i)=x_i,\qquad 1\le i\le N.
\]
By the definition \eqref{defnd0}, it follows that $\iota$ is isometric on $\{z_i\}_{i=1}^N$, namely,
\begin{align*}
d_{\GH}\bigl(\iota(z_{i_1}),\iota(z_{i_2})\bigr)
= d_{\GH}(x_{i_1},x_{i_2})
= d_T^Z(z_{i_1},z_{i_2}),
\qquad \forall\, i_1,i_2\in [1,N].
\end{align*}

In particular, $\{x_i\}_{i=1}^N$ is an $\ep$-separated subset of $X_{\GH}$. Since $(X_{\GH},d_{\GH})$ is compact, we conclude that $N$ must be finite. Hence every $\ep$-separated subset of $(Z_T,d_T^Z)$ is finite, which implies that $(Z_T,d_T^Z)$ is totally bounded. Since $(Z_T,d_T^Z)$ is complete, it follows that $(Z_T,d_T^Z)$ is compact, and in particular separable.

We may therefore choose a countable dense subset $\{z_i\}_{i\ge 1}\subset Z_T$. By the same argument as above, there exists a map
\[
\iota:\{z_i\}_{i\ge 1}\to X_{\GH}
\]
such that
\begin{align}\label{eq:subiso}
d_{\GH}\bigl(\iota(z_{i_1}),\iota(z_{i_2})\bigr)=d_T^Z(z_{i_1},z_{i_2}),
\qquad \forall\, i_1,i_2\ge 1.
\end{align}

Since $\{z_i\}_{i\ge 1}$ is dense in $Z_T$, \eqref{eq:subiso} allows us to extend $\iota$ uniquely to an isometric embedding from $(Z_T,d_T^Z)$ into $(X_{\GH},d_{\GH})$. This completes the proof.
\end{proof}

\subsection{Identification of the Gromov--Hausdorff limit}

Next, we prove that the isometric embedding $\iota$ is in fact surjective. Before doing so, we recall the neck decomposition theorem proved in \cite[Theorem 1.11]{fangli2025recti}. Although the theorem there is stated for four-dimensional noncollapsed Ricci flow limit spaces, an analogous neck decomposition also holds in dimension three. This can be seen either by taking the product of a three-dimensional Ricci flow limit space with a flat circle $S^1$, or by repeating the same proof verbatim in the three-dimensional setting. In fact, the proof is simpler in dimension three, since the only model space is the standard cylinder $S^2\times \R$; note that $\RP^2\times \R$ cannot occur because of the orientability assumption. We state the three-dimensional neck decomposition below. For the relevant definitions and notations, we refer the reader to \cite[Definition 5.9, Index]{fangli2025recti}.

\begin{thm}[3D neck decomposition theorem]\label{neckdecomgeneral3d}
Let $(Z, d_Z, \t)$ be a noncollapsed Ricci flow limit space arising as the pointed Gromov--Hausdorff limit of a sequence of orientable Ricci flows in $\MM(3, Y, T)$. For any constants $\delta>0$ and $\eta>0$, if $\zeta \le \zeta(Y, \delta, \eta)$, then the following holds.

Given $z_0 \in Z$ with $\t(z_0)-2 \zeta^{-2} r_0^2 \in \III^-$, we have the decomposition\textup{:}
	\begin{align*}
		&B^*(z_0,r_0)\subset \bigcup_a\big(\NNN_a\bigcap B^*(x_a,r_a)\big)\bigcup \bigcup_b B^*(x_b,r_b)\bigcup S^{1,\delta,\eta},\\
		&S^{1,\delta,\eta}\subset \bigcup_a\big(\CCC_{0,a}\bigcap B^*(x_a,r_a)\big)\bigcup\tilde{S}^{1,\delta,\eta},
	\end{align*}
with the following properties\textup{:}
	\begin{enumerate}[label=\textnormal{(\alph{*})}]
		\item For each $a$, $\NNN_a=B^*(x_a,2r_a)\setminus B^*_{r_x}(\CCC_a)$ is a $(1,\delta, \cc, r_a)$-cylindrical neck region, where $\cc=\cc(Y)$. 
		
		\item For each $b$, there exists a point in $B^*(x_b,2 r_b)$ which is $(2,\eta,r_b)$-symmetric.
		
		\item The following content estimates hold\textup{:}
		\begin{align*}
\sum_a r_a+\sum_b r_b+\HHH^1(S^{1,\delta,\eta})\leq C(Y) r_0 \quad \text{and} \quad \HHH^1(\tilde{S}^{1,\delta,\eta})=0.
	\end{align*}	
	\end{enumerate}
\end{thm}

Next, we prove the following key lemma.

\begin{lem}\label{lem:key}
For any $\ep>0$ and any sequence $\{x_j\}\subset M$, after passing to a subsequence of $\{t_j\}$ if necessary, the set $B_{g_{t_j}}(x_j,\ep)\times \{t_j\}$ contains an $H$-center of some point $z_j\in Z_T$ for all $j$, where $H=H(Y)>0$ is a constant.
\end{lem}

\begin{proof}
For small parameters $\delta>0$ and $\eta>0$, to be chosen later, let $\zeta=\zeta(Y,\delta,\eta)$ be such that Theorem~\ref{neckdecomgeneral3d} applies to $(Z,d_Z,\mathfrak t)$. By \cite[Equation (9.1)]{fangli2025structure}, the diameter of $(Z,d_Z)$ is bounded by a constant depending on $T$ and $\mathrm{diam}_{g_0}(M)$. Consequently, by \cite[Propositions 5.34 and 5.35]{fangli2025structure}, the set $Z_{[T/2,T]}$ can be covered by at most $C_0$ balls $B^*(z_i,r_0)$ with $z_i\in Z_{[T/2,T]}$ and
\[
r_0:=\zeta\sqrt{T/8},
\]
where $C_0$ depends on the Ricci flow, $\delta$, and $\eta$. Applying Theorem~\ref{neckdecomgeneral3d} to each ball $B^*(z_i,r_0)$ and using the fact that $S^{1,\delta,\eta}\subset \mathcal S\subset Z_T$, we obtain
\begin{align}\label{eq:estima0a}
 Z_{[T/2, T)} \subset \bigcup_a \lc\NNN_a \bigcap B^*(x_a, r_a) \rc \bigcup \bigcup_b B^*(x_b, r_b),
\end{align}
together with the content estimate
\begin{align}\label{eq:estima0}
\sum_a r_a+\sum_b r_b \le C_1,
\end{align}
where $C_1$ depends on the Ricci flow, $\delta$, and $\eta$.

By \eqref{eq:estima0}, we may choose a constant $\delta'\ll \ep$ such that
\begin{align}\label{eq:estima0c}
\sum_{r_a\le \delta'} r_a+\sum_{r_b\le \delta'} r_b \le \delta'.
\end{align}

From \cite[Lemma 6.21(v)]{fangli2025recti}, it follows that for any $\sigma>0$, if $\eta\le \eta(T,\sigma)$, then every point $y\in B^*(x_b,3r_b/2)$ is $(3,\sigma,\sigma r_b)$-symmetric. If $\sigma=\sigma(T)$ is chosen sufficiently small, then \cite[Theorem 8.13]{fangli2025structure} implies that
\[
r_{\Rm}^{-2}(y)\le C(T)r_b^{-2},
\]
where $r_{\Rm}$ denotes the curvature radius.

Let
\[
B_j:=B_{g_{t_j}}(x_j,\ep)\times \{t_j\}.
\]
We consider the following cases.

\textbf{Case 1.} There exist infinitely many $j$ such that
\[
B_j\bigcap \bigcup_{r_b>\delta'} B^*(x_b,r_b)\neq \emptyset.
\]

Passing to a subsequence of $\{t_j\}$ if necessary, we may assume that
\[
B_j\bigcap \bigcup_{r_b>\delta'} B^*(x_b,r_b)\neq \emptyset
\qquad \text{for all } j.
\]
By the discussion above, there exists a point $y_j\in B_j$ such that
\begin{align} \label{eq:estima1xa}
r_{\Rm}(y_j)\ge c(T)r_b\ge c(T)\delta'>0.
\end{align}
Hence, for all sufficiently large $j$, we can find a point $z_j\in Z_T\cap \mathcal R$ lying on the flow line of $\partial_{\mathfrak t}$ through $y_j$. By \eqref{eq:estima1xa}, it follows from \cite[Proposition 2.21(i)]{fangli2025structure} that $y_j$ is an $H(T)$-center of $z_j$. This proves the conclusion in this case.

\textbf{Case 2.} For all $j$,
\[
B_j \bigcap \bigcup_{r_b>\delta'} B^*(x_b,r_b)=\emptyset.
\]

In this case, \eqref{eq:estima0a} implies that
\begin{align*}
B_j \subset \bigcup_a \lc\NNN_a \bigcap B^*(x_a, r_a) \rc \bigcup \bigcup_{r_b \le \delta'} B^*(x_b, r_b).
\end{align*}
We decompose
\begin{align*}
B_j=\Omega_j^1\cup \Omega_j^2,
\end{align*}
where
\begin{align*}
\Omega^1_j:=  B_j \bigcap \lc \bigcup_{r_a \le \delta'} \lc\NNN_a \bigcap B^*(x_a, r_a) \rc \bigcup\bigcup_{r_b \le \delta'} B^*(x_b, r_b) \rc
\end{align*}
and
\begin{align*}
\Omega^2_j:=  B_j \bigcap \lc \bigcup_{r_a > \delta'} \lc\NNN_a \bigcap B^*(x_a, r_a) \rc \rc.
\end{align*}

By the same argument as in the proof of \cite[Theorem 6.31]{fangli2025recti}, we may choose $\delta=\delta(T)$, $\eta=\eta(T)$, and $\zeta=\zeta(T)$ such that
\begin{align}\label{eq:estima1}
\int_{\Omega_j^1} |\scal|\,\d V_{g_{t_j}}
\le C(T)\left(\sum_{r_a\le \delta'} r_a+\sum_{r_b\le \delta'} r_b\right)
\le C(T)\delta',
\end{align}
where the last inequality uses \eqref{eq:estima0c}.

Next, we analyze $\Omega_j^2$. Since the number of cylindrical neck regions $\NNN_a$ with $r_a>\delta'$ is finite by \eqref{eq:estima0}, we may assume for simplicity that there is only one such cylindrical neck region, denoted by $\NNN_a$; the general case follows by applying the same argument to each of the finitely many necks. Let $\CCC_a$ denote its center; for the relevant definitions, see \cite[Definition 5.9]{fangli2025recti}.

\textbf{Case 2a.} There exist infinitely many $j$ such that
\[
B^a_j:=B_j \bigcap \NNN_a \bigcap B^*(x_a, r_a) = \emptyset.
\]

Passing to a subsequence if necessary, we may assume that $B_j^a=\emptyset$ for all $j$. Therefore, by \eqref{eq:estima1},
\begin{align}\label{eq:estima1a}
\int_{B_j} |\scal|\,\d V_{g_{t_j}}
=\int_{\Omega_j^1} |\scal|\,\d V_{g_{t_j}}
\le C(T)\delta'.
\end{align}

\textbf{Case 2b.} For all $j$, we have $B_j^a\neq \emptyset$.

For any $x\in B_j^a$, choose $x'\in \CCC_a$ such that
\[
d_x:=d_Z(x,x')=d_Z(x,\CCC_a).
\]
If $x'\in Z_T$, then, provided $\delta\le \delta(T)$, the point $x$ is an $H$-center of $x'$ for some universal constant $H>0$. Indeed, by the definition of cylindrical neck regions, $x'$ is $(1,\delta,\cc^{-5}d_x)$-cylindrical with respect to $\mathcal L_{x',\cc^{-5}d_x}$, where $\cc=\cc(T)$; see \cite[Definition 5.9]{fangli2025recti}. If $x$ were not an $H$-center of $x'$ for sufficiently large $H$, then one could find a point $y\in \mathcal L_{x',\cc^{-5}d_x}$ such that
\[
d_Z(x,y)\le d_x/10.
\]
By \cite[Definition 5.9(n4)]{fangli2025recti}, this would imply the existence of a point $y'\in \CCC_a$ such that
\[
d_Z(y,y')\le d_x/10,
\]
and hence
\[
d_Z(x,y')\le d_x/5,
\]
contradicting the definition of $d_x$.

Thus, to complete the proof in this subcase, we may assume that $x'\in \CCC_{a} \setminus Z_T$ for every $x\in B_j^a$ and every $j$.

For each $j$, by a Vitali covering argument, we can find points $x_j^i\in B_j^a$ such that the balls $\{B^*(x_j^i,2 d_{x_j^i})\}_i$ are mutually disjoint and
\begin{align*}
B_j^a\subset \bigcup_i B^*(x_j^i,10 d_{x_j^i}).
\end{align*}

If for each $j$ there exists an index $i_j$ such that
\[
B^*((x_j^{i_j})',d_{x_j^{i_j}})\bigcap \CCC_{a} \bigcap Z_T\neq \emptyset,
\]
then the proof is complete. Indeed, in that case, set $m_j=x_j^{i_j}$ and $m_j'=(x_j^{i_j})'$, and choose
\[
y^j\in B^*(m_j',d_{m_j})\bigcap \CCC_{a} \bigcap Z_T.
\]
By \cite[Definition 3.2 and the open-ball characterization following it]{fangli2025structure}, this implies that for some $0<s_j<d_{m_j}$,
\begin{align*}
d_{W_1}^{g_{T-s_j^2}}
\bigl(\nu_{m_j';\,T-s_j^2},\nu_{y^j;\,T-s_j^2}\bigr)
<s_j.
\end{align*}
Using the monotonicity of $d^{g_t}_{W_1}$-distance and Bamler's non-expanding estimate (see
\cite[Theorem 9.1]{bamler2020structure}), it follows that $m_j$ is an $H(T)$-center of $y^j$, and hence the proof is complete.

Therefore, we may assume that
\[
B^*((x_j^i)',d_{x_j^i}) \bigcap \CCC_{a} \bigcap Z_T=\emptyset
\qquad \text{for all } i,j.
\]
As in the proof of \cite[Theorem 6.31]{fangli2025recti}; see \cite[Equation (6.58)]{fangli2025recti}, we then obtain
\begin{align*}
\int_{B_j^a} |\scal|\,\d V_{g_{t_j}}
\le C(T)\sum_i d_{x_j^i}
\le C(T)\sum_i \mu_a\bigl(B^*((x_j^i)',d_{x_j^i})\bigr),
\end{align*}
where $\mu_a$ is the packing measure of $\CCC_a$; see \cite[Definition 5.12]{fangli2025recti}. Since the balls $\{B^*((x_j^i)',d_{x_j^i})\}_i$ are mutually disjoint and also disjoint from $\CCC_{a}\cap Z_T$, we have
\begin{align}\label{eq:estima2a}
\int_{B_j^a} |\scal|\,\d V_{g_{t_j}}
\le C(T)\,\mu_a\bigl(\CCC_{a}\cap Z_{[2t_j-T,T)}\bigr)
=\Psi(j^{-1}),
\end{align}
where $\Psi(j^{-1})\to 0$ as $j\to \infty$. Combining \eqref{eq:estima1} and \eqref{eq:estima2a}, we obtain
\begin{align}\label{eq:estima3}
\int_{B_j} |\scal|\,\d V_{g_{t_j}}
&=\int_{\Omega_j^1} |\scal|\,\d V_{g_{t_j}}+\int_{B_j^a} |\scal|\,\d V_{g_{t_j}} \notag\\
&\le C(T)\delta'+\Psi(j^{-1}).
\end{align}

By \eqref{eq:estima1a} and \eqref{eq:estima3}, we have shown that in Case 2, the integral $\int_{B_j} |\scal|\,\d V_{g_{t_j}}$ can be made arbitrarily small by first choosing $\delta'$ sufficiently small and then taking $j$ sufficiently large.

If $B_{g_{t_j}}(x_j,\ep)=M$ for infinitely many $j$, then for any $z_0\in Z$, an $H_3$-center of $z_0$ at time $t_j$ lies in $B_{g_{t_j}}(x_j,\ep)$, and the proof is complete. Thus, in the remainder of the proof we may assume that
\[
B_{g_{t_j}}(x_j,\ep)\neq M
\qquad \text{for all } j.
\]

We claim that
\[
|\Rm|\le C_3
\qquad \text{on } B_{g_{t_j}}(x_j,\ep/2)
\]
for some constant $C_3$ independent of $j$. Indeed, if this were false, then by the same argument as in Case 2 of the proof of Proposition~\ref{prop:ballpacking}, we would obtain
\begin{align*}
\int_{B_j} |\scal|\,\d V_{g_{t_j}}\ge c_0\ep>0,
\end{align*}
where $c_0$ is independent of $j$. This contradicts the smallness of $\int_{B_j} |\scal|\,\d V_{g_{t_j}}$ established above.

Since $|\Rm|\le C_3$ on $B_{g_{t_j}}(x_j,\ep/2)$, we can argue exactly as in Case 1 above to obtain the desired conclusion.

This completes the proof.
\end{proof}

With the aid of Lemma~\ref{lem:key}, we now prove the surjectivity of the map $\iota$ constructed in Proposition~\ref{prop:embh}.

\begin{prop}\label{prop:surj}
The map $\iota$ constructed in Proposition~\ref{prop:embh} is surjective.
\end{prop}

\begin{proof}
We argue by contradiction. Since $(X_{\GH}, d_{\GH})$ is compact, $(Z_T,d_T^Z)$ is complete by \cite[Theorem 1.7(a)]{fangli2025structure}, and $\iota$ is an isometric embedding, it follows that $\iota(Z_T)$ is compact. Therefore, if $\iota(Z_T)\neq X_{\GH}$, then there exist $x^\infty\in X_{\GH}$ and $r_0>0$ such that
\begin{align}\label{eq:contr1}
B(x^\infty,2r_0)\cap \iota(Z_T)=\emptyset.
\end{align}

By the convergence \eqref{eq:grsequence}, we can find points $x_j\in M$ such that $(x_j,t_j)$ converges to $x^\infty$ in the Gromov--Hausdorff sense. After passing to a subsequence of $\{t_j\}$ if necessary, Lemma~\ref{lem:key} implies that for each $j$ there exists a point $y_j\in B_{g_{t_j}}(x_j,r_0)$ such that $(y_j,t_j)$ is an $H(T)$-center of some point $w_j\in Z_T$.

By the construction of $\iota$, there exists a countable dense subset $\{z_i\}_{i\ge 1}\subset Z_T$ such that $\iota(z_i)$ is defined as the Gromov--Hausdorff limit of a sequence $x_{i,j}^*\in M\times \{t_j\}$, where $x_{i,j}^*$ is an $H_3$-center of $z_i$. Since $(Z_T,d_T^Z)$ is compact, there exists a point $z_{i_0}$ such that
\begin{align}\label{eq:density2}
d_T^Z(w_j,z_{i_0})\le r_0/2
\end{align}
for infinitely many $j$. Passing to a further subsequence, we may assume that \eqref{eq:density2} holds for all $j$.

By the monotonicity estimate in \cite[Proposition 2.12]{fangli2025structure}, we have
\begin{align*}
d_{g_{t_j}}(y_j,x_{i_0,j}^*)
&\le d_{W_1}^{g_{t_j}}\bigl(\nu_{z_{i_0};t_j},\nu_{w_j;t_j}\bigr)+C(T)\sqrt{T-t_j} \\
&\le r_0/2+C(T)\sqrt{T-t_j}.
\end{align*}
It follows that any Gromov--Hausdorff limit of $(y_j,t_j)$, denoted by $y^\infty$, satisfies
\begin{align*}
d_{\GH}\bigl(y^\infty,\iota(z_{i_0})\bigr)\le r_0/2.
\end{align*}
On the other hand, since $y_j\in B_{g_{t_j}}(x_j,r_0)$, we obtain
\begin{align*}
d_{\GH}\bigl(x^\infty,\iota(z_{i_0})\bigr)\le 3r_0/2,
\end{align*}
which contradicts \eqref{eq:contr1}. This contradiction completes the proof.
\end{proof}

Combining Propositions~\ref{prop:embh} and~\ref{prop:surj}, we obtain the following theorem.

\begin{thm}\label{thm:convergence}
We have the Gromov--Hausdorff convergence
\begin{align*}
   (M, d_{g_t}) \longright[t \nearrow T]{\GHconvtext} (Z_T, d^Z_T). 
\end{align*}  
\end{thm}

Note that Theorem~\ref{thm:convergence} remains valid even when \(M\) is non-orientable. The only difference is that, in the proof of Proposition~\ref{prop:ballpacking}, a \(\delta\)-neck may be modeled on \(\mathbb{RP}^2\times \mathbb R\). Similarly, in the proof of Lemma~\ref{lem:key}, one must also allow quotient cylindrical neck regions modeled on \(\mathbb{RP}^2\times \mathbb R\). The rest of the argument is unchanged.

\subsection{Structure of the singular set at the first singular time}

The metric space $(Z_T,d_T^Z)$ admits the following regular-singular decomposition:
\begin{align*}
Z_T=\mathcal R_T \sqcup \mathcal S_T.
\end{align*}
Note that $\mathcal S_T=\mathcal S$, since all singular points are contained in $Z_T$.

\begin{thm}\label{thm:singular}
The singular set $\mathcal S_T$ satisfies the following properties:
\begin{enumerate}[label=\textnormal{(\arabic*)}]
\item There exists a constant $C>0$, depending on the Ricci flow, such that for all $r>0$,
\begin{equation}\label{eq:volume}
\bigl|\{y\in Z_T \mid d_T^Z(y,\mathcal S_T)<r\}\bigr|_T\le Cr^2.
\end{equation}

\item The Minkowski dimension of $\mathcal S_T$ with respect to $d_T^Z$ satisfies
\begin{equation}\label{eq:dimension}
\dim_{\mathscr M}(\mathcal S_T)\le 1.
\end{equation}
\end{enumerate}
\end{thm}

\begin{proof}
(1): By \cite[Equation (9.1), Propositions 5.34 and 5.35]{fangli2025structure}, the set $Z_{[T/2,T]}$ can be covered by at most $C_0$ balls $B^*(z_i,r_0)$ with $z_i\in Z_{[T/2,T]}$ and $r_0:=\sqrt{T}/10$, where $C_0$ is a constant depending on the Ricci flow.

For each ball $B^*(z_i,r_0)$, \cite[Corollary 6.28]{fangli2025recti} yields
\begin{align*}
\bigl|\{r_{\Rm}<r\}\cap B^*(z_i,r_0)\cap Z_T\bigr|_T \le C(T)r^2.
\end{align*}
Taking the union over all such balls, we obtain
\begin{align}\label{eq:volume1}
\bigl|\{r_{\Rm}<r\}\cap Z_T\bigr|_T \le C(T)C_0r^2.
\end{align}

By \cite[Lemma 6.4]{fangli2025structure},
\begin{align*}
\{y\in Z_T \mid d_T^Z(y,\mathcal S_T)<r\}
\subset S':=\{y\in Z_T \mid d_Z(y,\mathcal S)<r\}.
\end{align*}
Since $r_{\Rm}=0$ on $\mathcal S$, it follows from \cite[Proposition 7.6]{fangli2025structure} that every $y\in S'$ satisfies
\begin{align*}
r_{\Rm}(y)<C(Y)r.
\end{align*}
Hence
\begin{align*}
S'\subset \{y\in Z_T \mid r_{\Rm}(y)<C(Y)r\}.
\end{align*}
Therefore, \eqref{eq:volume} follows from \eqref{eq:volume1}.

(2): It suffices to show that for every sufficiently small $r>0$, the set $\mathcal S_T$ can be covered by at most $Cr^{-1}$ balls of radius $r$ with respect to $d_T^Z$.

To this end, fix a time $t$ and $r>0$, and define
\begin{align*}
S_r(t):=\{(x,t)\mid r_{\Rm}(x,t)<\delta r\},
\end{align*}
where $\delta>0$ is a small parameter to be chosen. By the same argument as in Case 2 of the proof of Proposition~\ref{prop:ballpacking}, we can choose $\delta$ sufficiently small, depending only on the Ricci flow, such that for every $x_0\in S_r(t)$,
\begin{align*}
\int_{B_{g_t}(x_0,r)} |\scal|\,\d V_{g_t} \ge c_0r>0,
\end{align*}
where $c_0>0$ depends only on the Ricci flow and is independent of $t$. In view of Theorem~\ref{thm:l1bound}, a standard covering argument then shows that $S_r(t)$ can be covered by at most $C_0r^{-1}$ balls of radius $r$, where $C_0$ depends only on the Ricci flow.

Finally, by the Gromov--Hausdorff convergence in Theorem~\ref{thm:convergence}, it follows that $\mathcal S_T$ can also be covered by at most $C_0r^{-1}$ balls of radius $r$ with respect to $d_T^Z$. This proves \eqref{eq:dimension}.
\end{proof}

\section{Heat kernel estimates on singular Ricci flows} \label{sec:4}

To extend the results of Section~\ref{sec:first} to singular Ricci flows, one needs to establish the corresponding heat kernel estimates. Such estimates are standard and well known for closed Ricci flows. Therefore, we focus here on the new aspects arising in the singular setting, and refer to \cite{bamler2020heatkernel,li2025heatnotes} for arguments that are analogous to the closed case.

Throughout this section, let $(\mathcal M,\mathfrak t,\partial_{\mathfrak t},g)$ be a singular Ricci flow with a normalized initial condition. We denote by $K(\cdot;\cdot)$ the heat kernel on $\mathcal M$. We also define the conjugate heat kernel measure on the singular Ricci flow by
\begin{align*}
\nu_{x;t}=K(x;\cdot)\,\d V_{g_t}(\cdot).
\end{align*}
Moreover, we define the \textbf{potential function} $f_x$ on $\mathcal M_{<\t(x)}$ by
\begin{align*}
K(x;\cdot)=\frac{1}{(4\pi(\mathfrak t(x)-\mathfrak t(\cdot)))^{3/2}}e^{-f_x(\cdot)}.
\end{align*}

\subsection{Maximum principles and gradient estimates}

We begin with the following maximum principle; compare with \cite[Lemma 5.1]{kleinerlott2017singular}.

\begin{lem}[Maximum principle I]\label{lem:maximum}
Suppose that $u$ is a smooth function and $X$ is a vector field on $\mathcal M_{[a,b]}$ such that $|u|+|X|\le C$. Assume that
\begin{align*}
(\partial_{\mathfrak t}-\Delta+X)u\le 0
\end{align*}
on $\mathcal M_{[a,b]}$, and that $u\le 0$ on $\mathcal M_a$. Then $u\le 0$ on $\mathcal M_{[a,b]}$.
\end{lem}

\begin{proof}
Set
\[
\square_X:=\partial_{\mathfrak t}-\Delta+X,
\]
and choose a large constant $L>0$ to be determined later. Since $X$ is bounded, we compute
\begin{align*}
\square_X(\scal+Le^{\mathfrak t-a}) \ge 2|\Ric|^2-C|\nabla \scal|+L
\end{align*}
on $\mathcal M_{[a,b]}$. Since $\mathcal M$ satisfies the canonical neighborhood assumption, \cite[Lemma 5.14]{kleinerlott2017singular} implies that there exists a constant $r_0>0$ such that every $x\in \mathcal M_{[a,b]}$ satisfies
\begin{align}\label{eq:max001}
|\nabla \scal(x)| \le C\scal(x)^{3/2}
\end{align}
if $\scal(x)\ge r_0^{-2}$, and
\begin{align}\label{eq:max002}
|\nabla \scal(x)| \le Cr_0^{-3}
\end{align}
if $\scal(x)<r_0^{-2}$. Here, $C$ in both \eqref{eq:max001} and \eqref{eq:max002} denotes a universal constant. In the first case, by choosing $r_0$ smaller if necessary, we may arrange that
\begin{align*}
2|\Ric|^2 \ge C\scal^{3/2}
\end{align*}
at such points. In the second case, we choose $L$ sufficiently large so that
\[
L\ge Cr_0^{-3}+1.
\]
Hence we may choose $L>0$ such that
\begin{align}\label{eq:max003}
\square_X(\scal+Le^{\mathfrak t-a}) \ge 1 \qquad \text{on } \mathcal M_{[a,b]},
\end{align}
and
\begin{align}\label{eq:max003a}
\scal+L\ge 1 \qquad \text{on } \mathcal M_a.
\end{align}

For any $\ep>0$, it follows from \eqref{eq:max003} that
\begin{align}\label{eq:max004}
\square_X\bigl(u-\ep(\scal+Le^{\mathfrak t-a})\bigr)\le -\ep.
\end{align}

Set
\[
u_\ep:=u-\ep(\scal+Le^{\mathfrak t-a}).
\]
By \eqref{eq:max003a}, we have $u_\ep<0$ on $\mathcal M_a$. We claim that $u_\ep\le 0$ on $\mathcal M_{[a,b]}$. Indeed, otherwise $u_\ep$ would attain a positive maximum at some point $x_0\in \mathcal M_{(a,b]}$, since $u$ is bounded and therefore
\begin{align*}
\limsup_{\scal(x)\to +\infty} u_\ep(x)=-\infty.
\end{align*}
However, at such a maximum point $x_0$, we have $\square_Xu_\ep\ge 0$, contradicting \eqref{eq:max004}.

Letting $\ep\searrow 0$, we conclude that $u\le 0$ on $\mathcal M_{[a,b]}$.
\end{proof}

\begin{rem}\label{rem:anothercond}
As is clear from the proof of the previous lemma, the assumption that $u$ is bounded on $\mathcal M_{[a,b]}$ can be replaced by the weaker condition
\begin{align*}
\limsup_{\scal(x)\to +\infty}\frac{|u(x)|}{\scal(x)}=0.
\end{align*}
\end{rem}

As an application of Lemma~\ref{lem:maximum}, we obtain the following estimate.

\begin{lem}[Gradient estimate I]\label{lem:grad1}
Suppose $u$ is a bounded smooth solution of the heat equation $\square u=0$ on $\mathcal M_{[a,b]}$, and assume that
\[
|\nabla u|\le 1 \qquad \text{on } \mathcal M_a.
\]
Then
\begin{align*}
|\nabla u|\le 1 \qquad \text{on } \mathcal M_{[a,b]}.
\end{align*}
\end{lem}

\begin{proof}
First, we have
\[
\square |\nabla u|\le 0
\]
on $\mathcal M_{[a,b]}$. Moreover, if $\scal(x)$ is sufficiently large, then by the canonical neighborhood assumption and parabolic regularity,
\[
|\nabla u(x)|^2\le C\,\scal(x).
\]
Therefore, the conclusion follows from Lemma~\ref{lem:maximum} and Remark~\ref{rem:anothercond}.
\end{proof}

Next, we generalize the following estimate from \cite{zhang2006some}.

\begin{lem}[Gradient estimate II]\label{lem:grad2}
Suppose $u$ is a smooth solution of the heat equation $\square u=0$ on $\mathcal M_{[a,b]}$ such that $0<u<1$. Then
\begin{align*}
\frac{|\nabla u|}{u}\le \sqrt{\frac{1}{\mathfrak t-a}}\,\sqrt{\log \frac{1}{u}}.
\end{align*}
\end{lem}

\begin{proof}
Without loss of generality, we may assume that
\[
\delta\le u\le 1-\delta
\]
for some constant $\delta>0$. Indeed, we may replace $u$ by $u(1-2\delta)+\delta$ and then let $\delta\searrow 0$.

Since $u$ is bounded, standard parabolic regularity together with the same argument as in the proof of Lemma~\ref{lem:grad1} implies that
\begin{align}\label{eq:grad001}
|\nabla u|^2\le C_\delta(1+\scal)\qquad \text{on } \mathcal M_{[a+\delta,b]}.
\end{align}

Next, we have (see \cite[Proposition 6.1]{li2025heatnotes})
\begin{align*}
\square \left((\mathfrak t-a-\delta)\frac{|\nabla u|^2}{u}-u\log\frac{1}{u}\right)
=
-\frac{2(\mathfrak t-a-\delta)}{u}\left|\nabla^2u-\frac{\d u\otimes \d u}{u}\right|^2
\le 0.
\end{align*}
Since $\delta\le u\le 1-\delta$, it follows from \eqref{eq:grad001}, Lemma~\ref{lem:maximum}, and Remark~\ref{rem:anothercond} that
\begin{align*}
(\mathfrak t-a-\delta)\frac{|\nabla u|^2}{u}\le u\log\frac{1}{u}
\qquad \text{on } \mathcal M_{[a+\delta,b]}.
\end{align*}
Letting $\delta\searrow 0$ yields the desired conclusion.
\end{proof}

Next, we generalize the gradient estimate in \cite[Theorem 4.1]{bamler2020heatkernel} to our setting. We recall the function
\begin{equation*}
\Phi(x)=\int_{-\infty}^x (4\pi)^{-1/2}e^{-t^2/4}\,\d t.
\end{equation*}
Note that $\Phi_t(x):=\Phi(t^{-1/2}x)$ is a solution to the one-dimensional heat equation with initial condition~$1_{[0,\infty)}$.

\begin{lem}[Gradient estimate III]\label{lem:grad3}
Suppose $u$ is a smooth solution of the heat equation $\square u=0$ on $\mathcal M_{[a,b]}$ such that $0<u<1$. Given a constant $L\ge 0$, assume that
\[
|\nabla \Phi_L^{-1}(u)|\le 1 \qquad \text{on } \mathcal M_a
\]
if $L>0$. Then for any $t\in [a,b]$,
\begin{align*}
|\nabla \Phi_{L+t-a}^{-1}(u)|\le 1 \qquad \text{on } \mathcal M_t.
\end{align*}
\end{lem}

\begin{proof}
As before, we may assume that $\delta\le u\le 1-\delta$ and $L>0$. Write
\[
u(x)=\Phi_{L+\t(x)-a}\bigl(h(x)\bigr).
\]
Then both $h$ and $|\nabla h|$ are uniformly bounded on $\mathcal M_{[a,b]}$. A direct computation shows that
\begin{align}\label{eq:grad002}
\square |\nabla h|^2
&=
-2|\nabla^2 h|^2
-\frac{1}{L+\mathfrak t-a}\langle \nabla (h^2),\nabla |\nabla h|^2\rangle
+\frac{1}{2(L+\mathfrak t-a)}(1-|\nabla h|^2)|\nabla h|^2.
\end{align}
Therefore, if we set
\[
v:=(|\nabla h|^2-1)_+,
\]
then \eqref{eq:grad002} implies that
\begin{align*}
\square v+\frac{1}{L+\mathfrak t-a}\langle \nabla h^2,\nabla v\rangle \le 0.
\end{align*}
Applying Lemma~\ref{lem:maximum}, we conclude that $v\equiv 0$ on $\mathcal M_{[a,b]}$, and hence the desired estimate follows.
\end{proof}

\subsection{Branches, Wasserstein distances, and variance monotonicity}

Next, we introduce the following definition.

\begin{defn}[Branch]
Let $\mathcal M$ be a singular Ricci flow, and let $b>0$ be a constant. A \textbf{branch} is a subset $\mathcal M'\subset \mathcal M_{[0,b]}$ such that for every $t\in [0,b]$, the slice $\mathcal M'_t$ is a connected component of $\mathcal M_t$. Note that a branch is uniquely determined by its final time-slice $\mathcal M'_b$; see \cite[Proposition 5.17]{kleinerlott2017singular}.

For a branch $\mathcal M'$ and $t\in [0,b]$, we denote by $(\overline{\mathcal M'_t},d_{g_t})$ the metric completion of $\mathcal M'_t$ with respect to $d_{g_t}$. By the canonical neighborhood assumption, this completion adds only finitely many points, each corresponding to an $\ep$-horn in $\mathcal M'_t$.

In particular, for any $x_0\in \mathcal M$, there exists a unique branch $\mathcal M'$, called the branch \textbf{with respect to $x_0$}, such that $\mathcal M'_t$ contains the support of $\nu_{x_0;t}$.
\end{defn}

It is clear that Lemmas~\ref{lem:maximum},~\ref{lem:grad1},~\ref{lem:grad2}, and~\ref{lem:grad3} extend verbatim to any branch of $\mathcal M$.

\begin{defn}[Wasserstein distance and variance]\label{def:w1}
For any $x,y\in \mathcal M$ and $t\le \min\{\mathfrak t(x),\mathfrak t(y)\}$, we define the \textbf{$W_1$-Wasserstein distance} and the \textbf{variance} associated with the conjugate heat kernels based at $x$ and $y$ as follows.

Let $\mathcal M^x$ and $\mathcal M^y$ be branches containing $x$ and $y$, respectively. If $\mathcal M_t^x=\mathcal M_t^y$, then we define $d_{W_1}^{g_t}(\nu_{x;t},\nu_{y;t})$ to be the usual $W_1$-Wasserstein distance on $(\overline{\mathcal M_t^x},d_{g_t})$. If $\mathcal M_t^x\neq \mathcal M_t^y$, then we define
\[
d_{W_1}^{g_t}(\nu_{x;t},\nu_{y;t})=+\infty.
\]

The variance between two probability measures $\nu_{x;t}$ and $\nu_{y;t}$ at time $t$ is defined by
\begin{equation*}
\Var_t(\nu_{x;t},\nu_{y;t})
:=\int_{\mathcal M_t}\int_{\mathcal M_t} d_{g_t}^2(x_1,x_2)\,\d \nu_{x;t}(x_1)\,\d \nu_{y;t}(x_2).
\end{equation*}
Here we set $d_{g_t}(x_1,x_2)=+\infty$ if $x_1$ and $x_2$ lie in different connected components of $\mathcal M_t$. For simplicity, we write
\[
\Var_t(\nu_{x;t})=\Var_t(\nu_{x;t},\nu_{x;t}).
\]
\end{defn}

By Lemma~\ref{lem:grad1}, one obtains the following monotonicity formula exactly as in the smooth case; cf.\ \cite[Lemma 2.7]{bamler2020heatkernel}.

\begin{lem}[Monotonicity I]\label{lem:mono1}
For any $x,y\in \mathcal M$ and $t\le \min\{\mathfrak t(x),\mathfrak t(y)\}$, the quantity
\begin{align*}
d_{W_1}^{g_t}(\nu_{x;t},\nu_{y;t})
\end{align*}
is nondecreasing in $t$. In particular, if $\mathfrak t(x)=\mathfrak t(y)=t_0$, then
\begin{align*}
d_{W_1}^{g_t}(\nu_{x;t},\nu_{y;t})\le d_{g_{t_0}}(x,y).
\end{align*}
\end{lem}

Next, we prove the geodesic convexity of each connected component of a time-slice of $\mathcal M$.

\begin{prop}[Geodesic convexity]\label{prop:geodesicconvex}
For each $T>0$, there exists a constant $r_0=r_0(T)>0$ such that the following holds. Define
\begin{align*}
\Omega_r(t):=\{x\in \mathcal M_t \mid \scal(x)\le r^{-2}\}.
\end{align*}
Suppose that $\mathcal M_t'$ is a connected component of $\mathcal M_t$ with $t\le T$. Then for any $r\le r_0$, there exists $r'<r$ such that any two points in $\Omega_r(t)\cap \mathcal M_t'$ can be connected by a minimizing geodesic contained in $\Omega_{r'}(t)\cap \mathcal M_t'$.

In particular, $\mathcal M_t'$ is geodesically convex.
\end{prop}

\begin{proof}
Fix $t\le T$, a connected component $\mathcal M_t'$, and $r\le r_0$, and set
\[
\Omega:=\mathcal M_t'\cap \Omega_r(t).
\]
It suffices to prove that there exists a compact set $K\subset \mathcal M_t'$ containing $\Omega$ such that any two points in $\Omega$ can be connected by a minimizing geodesic contained in $K$. We may assume that $\mathcal M_t'$ is noncompact, since otherwise we may simply take $K=\mathcal M_t'$.

By the canonical neighborhood assumption, if $r_0$ is chosen sufficiently small, then every connected component of $\mathcal M_t'\setminus \Omega$ is one of the following:
\begin{enumerate}[label=\textnormal{(\Alph*)}]
\item an $\ep$-tube whose two boundary components lie in $\partial \Omega$;
\item an $\ep$-horn whose unique boundary component lies in $\partial \Omega$;
\item an $\ep$-cap whose unique boundary component lies in $\partial \Omega$.
\end{enumerate}
See \cite[Definition 58.2]{kleinerlott2008notes} for the relevant definitions. In each case, the boundary of a component of $\mathcal M_t'\setminus \Omega$ consists of central spheres of $\ep$-necks.

We now enlarge $\Omega$ by adjoining all components of type \textnormal{(A)} and \textnormal{(C)} above, and denote the resulting compact set by $\Omega'$. Note that every component of $\partial \Omega'$ is the center of an $\ep$-neck, and the intersection of $\Omega'$ with such an $\ep$-neck corresponds to $(-\ep^{-1},0]\times S^2\subset \R\times S^2$. We then enlarge $\Omega'$ to a compact set $K$ so that its intersection with each such $\ep$-neck corresponds to
\[
(-\ep^{-1},10^{10}]\times S^2.
\]
It is clear that every connected component of $\mathcal M_t'\setminus K$ is an $\ep$-horn.

It remains to show that this set $K$ has the desired property. Suppose not. Then there exist points $x_0,y_0\in \Omega$ such that no minimizing geodesic joining $x_0$ and $y_0$ is contained in $K$. By definition, there exists a sequence of curves $\gamma_i(s)$ in $\mathcal M_t'$, parametrized by arclength, joining $x_0$ and $y_0$ such that
\begin{align*}
S_i:=\mathrm{length}(\gamma_i)\to d_{g_t}(x_0,y_0).
\end{align*}

If, after passing to a subsequence, the curves $\gamma_i$ are all contained in $K$, then by taking a limit we obtain a minimizing geodesic $\gamma_\infty$ in $K$ connecting $x_0$ and $y_0$, contradicting our assumption. Hence we may assume that $\gamma_i$ is not contained in $K$ for any $i$.

\begin{figure}[tbp]
    \centering
    \includegraphics[width=0.70\linewidth,trim=0 65 0 90,clip]{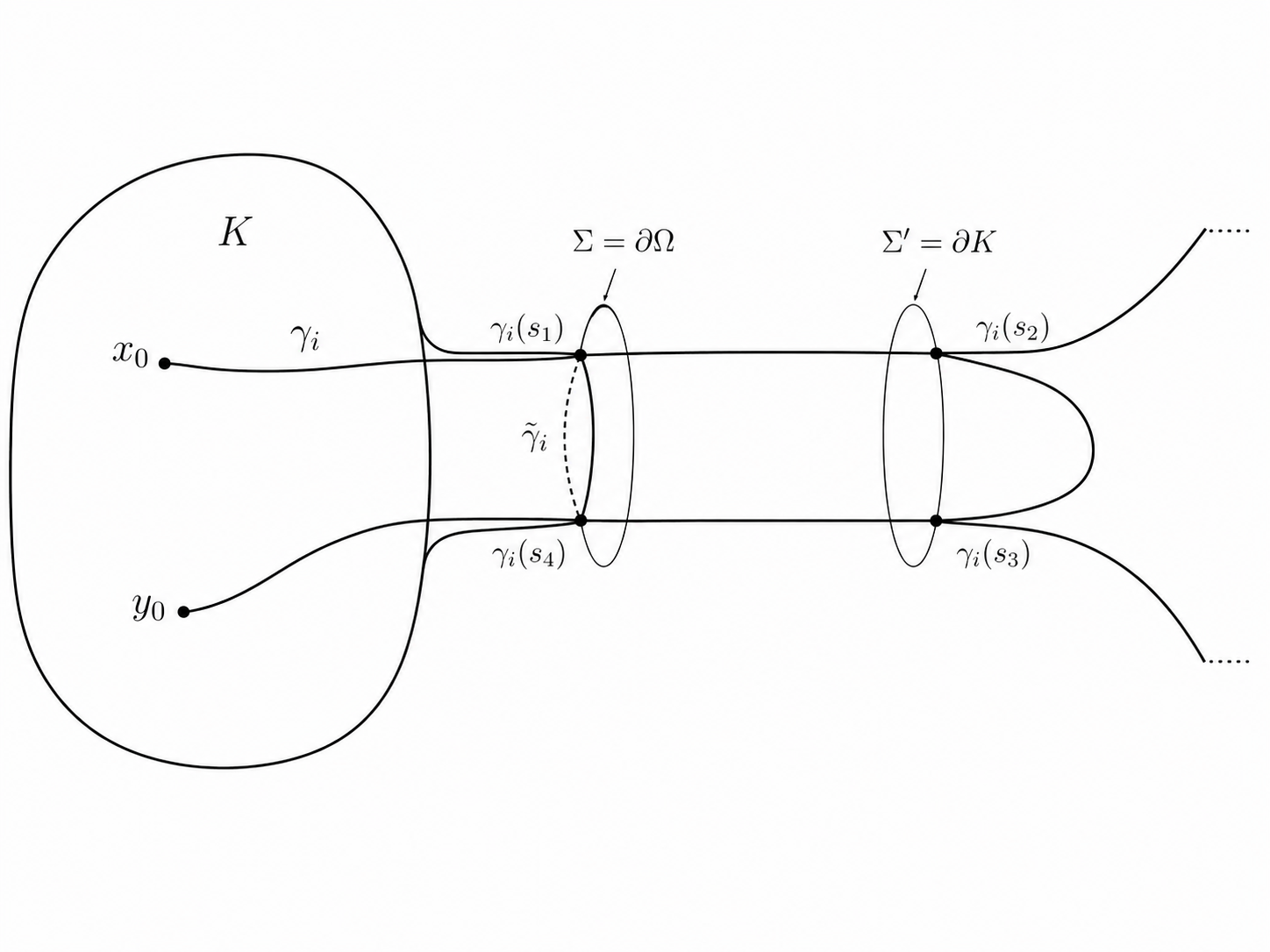}
    \caption{A shorter curve $\tilde\gamma_i$.}
    \label{fig:shorter-curve}
\end{figure}

For each $i$, by the construction of $K$, there exist parameters
\[
0\le s_1<s_2<s_3<s_4\le S_i
\]
such that $\gamma_i(s_1)$ and $\gamma_i(s_4)$ lie on some component $\Sigma$ of $\partial \Omega$, while $\gamma_i(s_2)$ and $\gamma_i(s_3)$ lie on the corresponding component $\Sigma'$ of $\partial K$. Moreover, $\gamma_i$, when restricted to $[0, s_2] \cup [s_3, S_i]$, is contained in $K$. Here, $\Sigma$ and $\Sigma'$ correspond to $\{0\}\times S^2$ and $\{10^{10}\}\times S^2$, respectively, inside an $\ep$-neck. Comparing with the standard cylinder, we see that
\begin{align*}
(s_2-s_1)+(s_4-s_3)\ge 10^{10}\sigma,
\end{align*}
where $\sigma$ is the scale of the $\ep$-neck. On the other hand, the points $\gamma_i(s_1)$ and $\gamma_i(s_4)$ can be joined by a curve lying entirely in $\Sigma$ whose length is at most $10^5\sigma$. Replacing the portion of $\gamma_i$ between $\gamma_i(s_1)$ and $\gamma_i(s_4)$ by this shorter curve, we obtain a new curve $\tilde\gamma_i$ that is contained in $K$ and has strictly shorter length than $\gamma_i$; see Figure~\ref{fig:shorter-curve}.

As before, taking a limit of the curves $\tilde\gamma_i$ yields a minimizing geodesic in $K$, contradicting our assumption. This proves the proposition.
\end{proof}

\begin{lem}[Monotonicity II]\label{lem:mono2}
For any $x,y\in \mathcal M$ and any
\[
s<t\le \min\{\mathfrak t(x),\mathfrak t(y)\},
\]
we have
\begin{align*}
\Var_s(\nu_{x;s},\nu_{y;s})+H_3 s
\le
\Var_t(\nu_{x;t},\nu_{y;t})+H_3 t,
\end{align*}
where $H_3=\pi^2+4$.
\end{lem}

\begin{proof}
Fix $x_0,y_0\in \mathcal M$ and times
\[
s_0<t_0\le \min\{\mathfrak t(x_0),\mathfrak t(y_0)\}.
\]

If
\[
\Var_{t_0}(\nu_{x_0;t_0},\nu_{y_0;t_0})=+\infty,
\]
then the desired inequality is trivial. Hence we may assume that
\[
\Var_{t_0}(\nu_{x_0;t_0},\nu_{y_0;t_0})<+\infty.
\]
By the definition of the variance, this implies that the supports of
$\nu_{x_0;t_0}$ and $\nu_{y_0;t_0}$ lie in the same connected component of
$\mathcal M_{t_0}$. Denote this component by $\mathcal M'_{t_0}$, and let
\[
\mathcal M'\subset \mathcal M_{[s_0,t_0]}
\]
be the branch determined by $\mathcal M'_{t_0}$. By the uniqueness of branches
and by the semigroup property of the heat kernel, for every
$t\in [s_0,t_0]$ the measures $\nu_{x_0;t}$ and $\nu_{y_0;t}$ are supported in
$\mathcal M'_t$.

On the fiber product
\[
\mathcal P:=\{(p,q)\in \mathcal M'\times \mathcal M'\mid \mathfrak t(p)=\mathfrak t(q)\in [s_0,t_0]\},
\]
endowed on each time-slice with the product metric $g_t\oplus g_t$, the function
\[
(p,q)\longmapsto d_{g_t}^2(p,q)
\]
satisfies
\begin{align}\label{eq:mono2003}
(\partial_{\mathfrak t}-\Delta_p-\Delta_q)d_{g_t}^2(p,q)\ge -H_3
\end{align}
in the barrier and viscosity sense. Indeed, the proof of \cite[Theorem 3.5]{bamler2020heatkernel} is local: one applies the second variation formula along a minimizing geodesic between $p$ and $q$ at the fixed time $t$. Such a minimizing geodesic exists in $\mathcal M_t'$ by Proposition~\ref{prop:geodesicconvex}. 

Let $\phi\in C_c^\infty(\mathcal M_{s_0}'\times \mathcal M_{s_0}')$ satisfy
\[
0\le \phi(p,q)\le d_{g_{s_0}}^2(p,q).
\]
Define, for $(p,q)\in \mathcal M_t'\times \mathcal M_t'$ and $t\in [s_0,t_0]$,
\[
u(p,q,t):=
\iint_{\mathcal M_{s_0}'\times \mathcal M_{s_0}'}
\phi(a,b)K(p;a)K(q;b)\,\d V_{g_{s_0}}(a)\d V_{g_{s_0}}(b)-H_3(t-s_0),
\]
where $K(p;a)$ denotes the heat kernel from the spacetime point $p$ back to the point $a\in \mathcal M_{s_0}'$. Then
\[
(\partial_{\mathfrak t}-\Delta_p-\Delta_q)u=-H_3,
\qquad
u(\cdot,\cdot,s_0)=\phi.
\]
By \eqref{eq:mono2003}, the function $u-d_{g_t}^2$ is a viscosity subsolution of the heat equation on $\mathcal P$ and is nonpositive at time $s_0$. 

Applying the product maximum principle in the same way as Lemma~\ref{lem:maximum}, using the barrier $\scal(p)+\scal(q)+Le^{\mathfrak t-s_0}$, we obtain
\begin{align}\label{eq:mono2004}
u(p,q,t)\le d_{g_t}^2(p,q)
\qquad \text{for all }t\in [s_0,t_0].
\end{align}

Using the semigroup property of the heat kernel and \eqref{eq:mono2004}, we obtain
\begin{align}\label{eq:mono2005}
&\iint_{\mathcal M_{s_0}'\times \mathcal M_{s_0}'}
\phi\,\d\nu_{x_0;s_0}\,\d\nu_{y_0;s_0}+H_3s_0 \notag\\
={}&
\iint_{\mathcal M_{t_0}'\times \mathcal M_{t_0}'}
 u\,\d\nu_{x_0;t_0}\,\d\nu_{y_0;t_0}+H_3t_0 \notag\\
\le{}&
\iint_{\mathcal M_{t_0}'\times \mathcal M_{t_0}'}
 d_{g_{t_0}}^2\,\d\nu_{x_0;t_0}\,\d\nu_{y_0;t_0}+H_3t_0.
\end{align}
Finally, take an increasing sequence of nonnegative compactly supported smooth functions $\phi_k$ with $\phi_k\le d_{g_{s_0}}^2$ and
$\phi_k\nearrow d_{g_{s_0}}^2$ almost everywhere with respect to $\nu_{x_0;s_0}\otimes\nu_{y_0;s_0}$. Letting $k\to\infty$ in \eqref{eq:mono2005} and applying monotone convergence yields the desired inequality.
\end{proof}

The following corollary follows immediately from Lemma~\ref{lem:mono2} (cf. \cite[Corollary 3.8]{bamler2020heatkernel}).

\begin{cor}[$H_3$-concentration]\label{cor:concern}
For any $x\in \mathcal M$ and any $t\le \mathfrak t(x)$, we have
\begin{align*}
\Var_t(\nu_{x;t})\le H_3\bigl(\mathfrak t(x)-t\bigr).
\end{align*}
\end{cor}

As in \cite[Definition 3.10]{bamler2020heatkernel}, we introduce the following definition.

\begin{defn}
A point $z\in \mathcal M_t$ is called an \textbf{$H$-center} of $x_0\in \mathcal M$ for a constant $H>0$ if $t<\mathfrak t(x_0)$ and
\begin{align*}
\Var_t(\delta_z,\nu_{x_0;t})\le H\bigl(\mathfrak t(x_0)-t\bigr).
\end{align*}
By Corollary~\ref{cor:concern}, an $H_3$-center exists for every $t<\mathfrak t(x_0)$.
\end{defn}

The following result can be proved in the same way as \cite[Propositions 3.12 and 3.13]{bamler2020heatkernel}.

\begin{prop}\label{existenceHncenter}
Any two $H_3$-centers $z_1,z_2\in \mathcal M_t$ of $x_0$ satisfy
\begin{align*}
d_{g_t}(z_1,z_2)\le 2\sqrt{H_3\bigl(\mathfrak t(x_0)-t\bigr)}.
\end{align*}
Moreover, if $z\in \mathcal M_t$ is an $H$-center of $x_0$, then for any $L>0$,
\begin{align*}
\nu_{x_0;t}\Bigl(B_{g_t}\bigl(z,\sqrt{LH(\mathfrak t(x_0)-t)}\bigr)\Bigr)\ge 1-L^{-1}.
\end{align*}
\end{prop}

Next, we establish estimates for the potential function and higher-order estimates for the conjugate heat kernel, which will play an important role in justifying integration by parts.

\begin{lem}\label{lem:potent1}
For any $x_0\in \mathcal M$ and any $t<\mathfrak t(x_0)$, let $\mathcal M'$ be the branch associated with $x_0$. Suppose that $U\subset \mathcal M_t'$ is an $\ep$-neck of scale $r$ such that
\[
t\le \mathfrak t(x_0)-2r^2.
\]
Then the following estimates hold.
\begin{enumerate}[label=\textnormal{(\roman*)}]
\item For any $k\ge 0$ and $m\ge 0$, on $U$ we have 
\begin{align*}
|\nabla^k K(x_0;\cdot)| \le C_{m,k} r^{m-k},
\end{align*}
where $C_{m,k}$ is a constant depending on $x_0$, $m$, and $k$.

\item On $U$, we have
\begin{align*}
|\nabla f_{x_0}|^2 \le C r^{-2}(1+f_{x_0}),
\end{align*}
where $C$ is a constant depending on $x_0$.

\item For any $k\ge 1$, we have
\begin{align*}
|\nabla^k f_{x_0}| \le C_k r^{-k} e^{C_k\sqrt{f_{x_0}}},
\end{align*}
where $C_k$ is a constant depending on $k$ and $x_0$.
\end{enumerate}
\end{lem}

\begin{proof}
For simplicity, write
\[
w=K(x_0;\cdot),\qquad f=f_{x_0}.
\]

\textbf{(i)} Since $\square^*w=0$, the conclusion follows directly from standard parabolic estimates together with Theorem~\ref{thm:heatkernel2}.

\textbf{(ii)} By Theorem~\ref{thm:heatkernel}(c), we have $w\le C_0$ on $U$. Since $\square^*w=0$ on $\mathcal M_{[t,t+r^2]}$, it follows from \cite[Theorem 16.52]{chowetal2010part1-4} that
\begin{align*}
\frac{|\nabla w|^2}{w^2} \le C_1 r^{-2}\left(1+\log\frac{C_0}{w}\right)
\end{align*}
on $U$. In view of the definition of $f$, this proves \textnormal{(ii)}.

\textbf{(iii)} In the proof, all constants $Q_i$ depend on $x_0$.

Fix $z\in U$. By \textnormal{(ii)}, we have
\[
|\nabla \sqrt{f}| \le Q_1 r^{-1}.
\]
Hence there exists a constant $Q_2>1$ such that
\begin{align*}
f(z)-Q_2(\sqrt{f(z)}+1)\le f(y)\le f(z)+Q_2(\sqrt{f(z)}+1)
\end{align*}
for any $y\in B_{g_t}(z,2r)$. Therefore,
\begin{align}\label{eq:poten003}
Q_3^{-1}w(z)e^{-Q_2\sqrt{f(z)}}\le w(y)\le Q_3 w(z)e^{Q_2\sqrt{f(z)}}
\end{align}
for some constant $Q_3>1$ and all $y\in B_{g_t}(z,2r)$.

Since $\square^*w=0$, it follows from the parabolic Harnack inequality and \eqref{eq:poten003} that
\begin{align*}
w \le Q_4 e^{Q_4\sqrt{f(z)}} w(z)
\quad \text{on } B_{g_t}(z,r)\times [t,t+r^2].
\end{align*}
By the standard parabolic interior estimate, we then obtain
\begin{align*}
|\nabla^k w|(z)\le Q'_k r^{-k} e^{Q_4\sqrt{f(z)}} w(z),
\end{align*}
where $Q_k'$ is a constant depending on $x_0$ and $k$. Using the identity
\[
f=-\log w-\frac{3}{2}\log\bigl(4\pi(\mathfrak t(x_0)-t)\bigr)
\]
and differentiating repeatedly, we conclude by induction that
\begin{align*}
|\nabla^k f|(z)\le C_k r^{-k} e^{C_k\sqrt{f(z)}}.
\end{align*}
This proves \textnormal{(iii)}.
\end{proof}

Next, we prove the following identity, which will be used repeatedly.

\begin{lem}\label{usefuliden}
Suppose $u$ is a smooth spacetime function defined in a neighborhood of $\mathcal M_t$ such that
\[
|u|+|\nabla u| \le L(1+\scal^L)
\]
for some constant $L>1$. Then for any $x_0$ with $\mathfrak t(x_0)\ge t$, if both $|\square u|$ and $|\Delta u|$ are integrable with respect to $\d \nu_{x_0;t}$, then
\begin{align*}
\partial_{\mathfrak t}\int_{\mathcal M_t} u\,\d \nu_{x_0;t}
=
\int_{\mathcal M_t} \square u\,\d \nu_{x_0;t}.
\end{align*}
\end{lem}

\begin{proof}
For simplicity, write $w=K(x_0;\cdot)$.

We first show that the growth assumption on $u$ implies that $u$ is integrable with respect to $\d \nu_{x_0;t}$. Indeed, by Theorem~\ref{thm:heatkernel2}, we have
\[
w\le C\scal^{-2L}.
\]
On the other hand, the total volume of $\mathcal M_t$ is finite by \cite[Proposition 5.5(1)]{kleinerlott2017singular}. Hence the assumption on $u$ implies that
\begin{align*}
\int_{\mathcal M_t} |u|\,\d \nu_{x_0;t}<\infty.
\end{align*}
Similarly, Lemma~\ref{lem:potent1}(i) implies that
\begin{align*}
\int_{\mathcal M_t} |u\,\Delta w|\,\d V_{g_t}<\infty.
\end{align*}

A direct computation shows that for any domain $\Omega\subset\subset \mathcal M_t$,
\begin{align}\label{eq:useiden000}
\partial_{\mathfrak t}\int_{\Omega} u\,\d \nu_{x_0;t}
&=
\int_{\Omega} (\square u)w+(\Delta u)w-u(\Delta w)\,\d V_{g_t} \notag\\
&=
\int_{\Omega} \square u\,\d \nu_{x_0;t}
+\int_{\partial\Omega} \frac{\partial u}{\partial \vec n}w-u\frac{\partial w}{\partial \vec n}\,\d \sigma_t,
\end{align}
where $\vec n$ is the outward unit normal along $\partial\Omega$, and $\d \sigma_t$ is the induced area form on $\partial\Omega$.

As in the proof of Proposition~\ref{prop:geodesicconvex}, we can find an exhaustion of the branch $\mathcal M'$ associated with $x_0$ by domains $\Omega_i$ such that every component of $\partial\Omega_i$ is a central two-sphere
of an $\ep_i$-neck of scale $r_i$, where $r_i,\ep_i\to 0$. For each $i$, we have
\begin{align}\label{eq:useiden001}
&\left|\partial_{\mathfrak t}\int_{\mathcal M_t} u\,\d \nu_{x_0;t}
-\partial_{\mathfrak t}\int_{\Omega_i} u\,\d \nu_{x_0;t}\right|
+\left|\int_{\mathcal M_t\setminus \Omega_i} \square u\,\d \nu_{x_0;t}\right| \notag\\
\le\;&
\left|\int_{\mathcal M_t\setminus \Omega_i}
2|(\square u)|\,w+|(\Delta u)w|+|u(\Delta w)|\,\d V_{g_t}\right|
\xrightarrow[i\to\infty]{}0,
\end{align}
since
\[
|(\square u)|\,w+|(\Delta u)w|+|u(\Delta w)|
\]
is integrable with respect to $\d V_{g_t}$.

On the other hand,
\begin{align*}
\left|\int_{\partial\Omega_i} \frac{\partial u}{\partial \vec n}w-u\frac{\partial w}{\partial \vec n}\,\d \sigma_t\right|
\le
\int_{\partial\Omega_i} |\nabla u|\,w+|u|\,|\nabla w|\,\d \sigma_t.
\end{align*}
Note that $\sigma_t(\partial\Omega_i)\le Cr_i^2$. Using the assumption on $u$ together with Lemma~\ref{lem:potent1}(i), we obtain
\begin{align*}
\lim_{i\to\infty}
\left|\int_{\partial\Omega_i} \frac{\partial u}{\partial \vec n}w-u\frac{\partial w}{\partial \vec n}\,\d \sigma_t\right|
=0.
\end{align*}
Combining this with \eqref{eq:useiden000} and \eqref{eq:useiden001} yields the desired identity.
\end{proof}

Next, we prove the following result, which extends Theorem~\ref{thm:heatkernel}(b).

\begin{lem}\label{lem:delta}
For any $x_0\in \mathcal M$ and any $h\in C(\mathcal M'_{[a,\mathfrak t(x_0)]})$ satisfying
\[
|h| \le L(1+\scal^L)
\]
for some constant $L>1$, where $\mathcal M'$ is the branch associated with $x_0$, we have
\begin{equation*}
\lim_{t\nearrow \mathfrak t(x_0)}\int_{\mathcal M_t} h\,\d \nu_{x_0;t}=h(x_0).
\end{equation*}
\end{lem}

\begin{proof}
Set $t_0=\mathfrak t(x_0)$ and
\[
\Omega_r=\{x\in \mathcal M'_{t_0}\mid \scal(x)\le r^{-2}\}.
\]

As in the proof of Proposition~\ref{prop:geodesicconvex}, we may enlarge $\Omega_r$ slightly to a domain $\Omega_r'$ such that every component of $\partial \Omega_r'$ is the center of an $\ep$-neck of scale $r$. Using the worldlines through $\Omega_r'$, one constructs a product domain
\[
U_r=\Omega_r'\times [t_0-\delta_r,t_0].
\]
If $\delta_r$ is chosen sufficiently small, then for any $t\in [t_0-\delta_r,t_0]$, the slice $\Omega_r'\times \{t\}$ contains all points with $\scal\le 2r^{-2}$, and every boundary component is the center of an $\ep$-neck of scale in $[0.9r,1.1r]$. Similarly, one defines $\Omega_{r/2}$, $\Omega_{r/2}'$, and
\[
U_{r/2}=\Omega_{r/2}'\times [t_0-\delta_r,t_0],
\]
which contain $\Omega_r$, $\Omega_r'$, and $U_r$, respectively.

Choose a cutoff function $\psi$ supported in $U_{r/2}$ such that $\psi=1$ on $U_r$. Then Theorem~\ref{thm:heatkernel}(b) implies that
\begin{align}\label{eq:delta001}
\lim_{t\nearrow t_0}\int_{U_{r/2}\cap \mathcal M_t'} h\psi\,\d \nu_{x_0;t}=h(x_0).
\end{align}
On the other hand, by Theorem~\ref{thm:heatkernel2}, we have
\[
K(x_0;\cdot)\le C\scal^{-2L}
\qquad \text{on } \mathcal M_t'\setminus U_r.
\]
Therefore, using the growth assumption on $h$, we obtain
\begin{align*}
\lim_{r\to 0}\sup_{t\in [t_0-\delta_r,t_0]}
\int_{\mathcal M_t'\setminus U_r} |h|\,\d \nu_{x_0;t}=0.
\end{align*}
Combining this with \eqref{eq:delta001} completes the proof.
\end{proof}

With the help of Lemmas~\ref{lem:maximum},~\ref{lem:grad1},~\ref{lem:grad3},~\ref{usefuliden}, and~\ref{lem:delta}, we can follow the proof of \cite[Theorem 11.1]{bamler2020heatkernel} to obtain the following $L^p$-Poincar\'e inequality.

\begin{thm}[$L^p$-Poincar\'e inequality]
Given $x_0\in \mathcal M_{t_0}$, the following inequality holds for every $p\ge 1$:
\begin{align*}
\int_{\mathcal M_s} |u|^p\,\d \nu_{x_0;s} \le C(p)(t_0-s)^{\frac{p}{2}}\int_{\mathcal M_s} |\nabla u|^p\,\d \nu_{x_0;s},
\end{align*}
where $s < t_0$, $\mathcal M'$ is the branch associated with $x_0$, and $u$ is any function satisfying $u\in W^{1,p}(\mathcal M_s',\d \nu_{x_0;s})$ and
\[
\int_{\mathcal M_s'} u\,\d \nu_{x_0;s}=0.
\]
One may choose $C(1)=\sqrt{\pi}$ and $C(2)=2$.
\end{thm}

Similarly, using Lemmas~\ref{lem:maximum},~\ref{lem:grad2},~\ref{usefuliden}, and~\ref{lem:delta}, one can follow the proof of \cite[Theorem 1.10(2)]{heinnaber2014new} to derive the following logarithmic Sobolev inequality.

\begin{thm}[Logarithmic Sobolev inequality]\label{thm:logsob}
Given $x_0\in \mathcal M_{t_0}$, the following inequality holds:
\begin{align*}
\int_{\mathcal M_s} u\log u\,\d \nu_{x_0;s}
\le
(t_0-s)\int_{\mathcal M_s} \frac{|\nabla u|^2}{u}\,\d \nu_{x_0;s},
\end{align*}
where $s< t_0$, $\mathcal M'$ is the branch associated with $x_0$, and $u$ is any function such that $\sqrt{u}\in W^{1,2}(\mathcal M_s',\d \nu_{x_0;s})$ and
\[
\int_{\mathcal M_s'} u\,\d \nu_{x_0;s}=1.
\]
If equality holds, then either $u$ is constant or $(\mathcal M_s',g_s)$ splits off an $\R$-factor.
\end{thm}

An application of Theorem~\ref{thm:logsob} is the following Gaussian concentration estimate, proved in \cite{heinnaber2014new}. In our setting, one can follow the same argument as in \cite[Theorem 2.15(i)]{fangli2025structure}.

\begin{thm}[Gaussian concentration]\label{thm:Gaussian}
Given $x_0\in \mathcal M$, for any $\epsilon>0$, any $L>0$, and any $t< t_0=\mathfrak t(x_0)$, we have
\begin{align*}
\nu_{x_0;t}\bigl(\mathcal M_t\setminus B_{g_t}(z,L)\bigr)
\le
C(\ep)\exp\left(-\frac{L^2}{(4+\epsilon)(t_0-t)}\right),
\end{align*}
where $z\in \mathcal M_t$ is any $H_3$-center of $x_0$.
\end{thm}

\subsection{Pointed entropies on singular Ricci flows} \label{sub:entropy}

Next, we introduce the following pointed entropy quantities, originally defined in \cite{heinnaber2014new} for closed Ricci flows.

\begin{defn}
Perelman's $\mathcal W$-entropy and the Nash entropy\index{Nash entropy} based at $x_0\in \mathcal M_{t_0}$ are defined, for any $\tau\in (0,t_0]$, by
\begin{align*}
\mathcal W_{x_0}(\tau)
&=
\int_{\mathcal M_{t_0-\tau}}
\Bigl(
\tau(2\Delta f_{x_0}-|\nabla f_{x_0}|^2+\scal)+f_{x_0}-3
\Bigr)\,\d \nu_{x_0;t_0-\tau},\\
\mathcal N_{x_0}(\tau)
&=
\int_{\mathcal M_{t_0-\tau}} f_{x_0}\,\d \nu_{x_0;t_0-\tau}-\frac{3}{2},
\end{align*}
where $f_{x_0}$ is the potential function based at $x_0$.
\end{defn}

\begin{lem}\label{lem:welldefined}
For any $x_0\in \mathcal M_{t_0}$ and any $\tau\in (0,t_0]$, both $\mathcal N_{x_0}(\tau)$ and $\mathcal W_{x_0}(\tau)$ are well-defined and finite.
\end{lem}

\begin{proof}
On $\mathcal M_{t_0-\tau}$, the function $f_{x_0}$ is bounded from below; see Theorem~\ref{thm:heatkernel}(c). The finiteness of $\mathcal N_{x_0}(\tau)$ then follows from Theorem~\ref{thm:Gaussian} and a standard argument; see \cite[Proposition 9.2]{li2025heatnotes}.

By the differential Harnack inequality (see \cite[Corollary 5.8]{lai2021}), $\mathcal W_{x_0}(\tau)$ is well-defined and nonpositive. To show that it is finite, let $\{\Omega_i\}$ be a sequence of domains exhausting $\mathcal M'_{t_0-\tau}$, where $\mathcal M'$ is the branch associated with $x_0$, such that every component of $\partial \Omega_i$ is a central two-sphere
of an $\ep_i$-neck of scale $r_i$, with $r_i,\ep_i\to 0$.

Set $v=\tau(2\Delta f_{x_0}-|\nabla f_{x_0}|^2+\scal)+f_{x_0}-3$. Then we have
\begin{align}\label{eq:nash001}
\int_{\Omega_i} v\,\d \nu_{x_0;t_0-\tau}
&=
\int_{\Omega_i}
\Bigl(
\tau(|\nabla f_{x_0}|^2+\scal)+f_{x_0}-3
\Bigr)\,\d \nu_{x_0;t_0-\tau}
+2\tau \int_{\partial \Omega_i} \frac{\partial f_{x_0}}{\partial \vec n}K(x_0;\cdot)\,\d \sigma_{t_0-\tau}.
\end{align}

As in the proof of Lemma~\ref{usefuliden}, we estimate
\begin{align*}
\left|\int_{\partial \Omega_i} \frac{\partial f_{x_0}}{\partial \vec n}K(x_0;\cdot)\,\d \sigma_{t_0-\tau}\right|
&\le
\int_{\partial \Omega_i} |\nabla f_{x_0}|\,K(x_0;\cdot)\,\d \sigma_{t_0-\tau}\\
&\le
Cr_i \sup_{\partial \Omega_i} \sqrt{f_{x_0}+1}\,K(x_0;\cdot),
\end{align*}
where in the last inequality we used Lemma~\ref{lem:potent1}(ii).

By Theorem~\ref{thm:heatkernel2}, it follows that
\begin{align*}
\lim_{i\to\infty}
\int_{\partial \Omega_i} \frac{\partial f_{x_0}}{\partial \vec n}K(x_0;\cdot)\,\d \sigma_{t_0-\tau}
=0.
\end{align*}
Combining this with \eqref{eq:nash001}, we conclude that $\mathcal W_{x_0}(\tau)$ is finite.
\end{proof}

\begin{rem}
After integration by parts, justified as in the proof of Lemma~\ref{lem:welldefined}, we have
\begin{align*}
\mathcal W_{x_0}(\tau)
=
\int_{\mathcal M_{t_0-\tau}}
\Bigl(
\tau(|\nabla f_{x_0}|^2+\scal)+f_{x_0}-3
\Bigr)\,\d \nu_{x_0;t_0-\tau}.
\end{align*}
\end{rem}

Next, with the aid of Lemma~\ref{usefuliden}, one can follow the same proof as in the closed Ricci flow setting (see, for example, \cite[Proposition 5.2]{bamler2020heatkernel}) to obtain the following monotonicity properties.

\begin{prop}\label{prop:nash1}
For any $x_0\in \mathcal M_{t_0}$, the following statements hold.
\begin{enumerate}[label=\textnormal{(\alph*)}]
\item $\mathcal W_{x_0}(0):=\lim_{\tau \searrow 0} \mathcal W_{x_0}(\tau)=0$, and for any $\tau_0\in (0,t_0]$,
\begin{align*}
\mathcal W_{x_0}(\tau_0)
=
-2\int_0^{\tau_0}
\tau \int_{\mathcal M_{t_0-\tau}}
\left|
\Ric+\nabla^2 f_{x_0}-\frac{g}{2\tau}
\right|^2
\,\d \nu_{x_0;t_0-\tau}\,\d \tau.
\end{align*}
In particular, $\mathcal W_{x_0}(\tau)$ is nonincreasing in $\tau$.

\item $\mathcal N_{x_0}(0):=\lim_{\tau \searrow 0} \mathcal N_{x_0}(\tau)=0$, and for any $\tau_0\in (0,t_0]$,
\begin{align*}
\mathcal N_{x_0}(\tau_0)
=
\frac{1}{\tau_0}\int_0^{\tau_0} \mathcal W_{x_0}(\tau)\,\d \tau.
\end{align*}

\item For any $0<\tau_1\le \tau_2\le t_0$,
\begin{align*}
\mathcal W_{x_0}(\tau_2)\le \mathcal N_{x_0}(\tau_2)\le \mathcal N_{x_0}(\tau_1).
\end{align*}
\end{enumerate}
\end{prop}

We now recall the following definition, analogous to \cite[Definition 2.20]{fangli2025structure}.

\begin{defn}\label{def:entropybound}
A singular Ricci flow $\mathcal M$ is said to have entropy bounded below by $-Y$ over $[0,T]$ if
\begin{align*}
\inf_{x\in \mathcal M_{[0,T]}} \inf_{\tau>0} \mathcal N_x(\tau)\ge -Y,
\end{align*}
where the inner infimum is taken over all $\tau>0$ for which the Nash entropy $\mathcal N_x(\tau)$ is well-defined.
\end{defn}

Note that if a singular Ricci flow has a normalized initial condition, then Proposition~\ref{prop:nash1} implies, as in \eqref{eq:nashlower}:
\begin{align}\label{entropybd-Y1}
\inf_{\tau \in (0, \t(x)]}\mathcal N_x(\tau)
\ge
\mathcal N_x(\mathfrak t(x))
\ge
\inf_{\tau\in (0,T]} \boldsymbol{\mu}(g_0,\tau)
\ge -Y_0,
\end{align}
for any $x\in \mathcal M_{[0,T]}$, where $Y_0$ is a constant depending only on $T$.

Next, with the help of Lemmas~\ref{lem:grad3},~\ref{usefuliden}, and~\ref{lem:delta}, we can follow verbatim the proofs of \cite[Proposition 4.2, Theorem 5.9, Corollary 5.11]{bamler2020heatkernel} to obtain the following properties of the modified Nash entropy, defined by
\begin{align*}
\mathcal N_s^*(x):=\mathcal N_x(\mathfrak t(x)-s), \qquad \forall\, s\in [0,\mathfrak t(x)).
\end{align*}

\begin{thm}\label{thm:nash}
For any $s\ge 0$, if $\scal\ge \scal_{\min}$ on $\mathcal M_s$, then the following properties hold.
\begin{enumerate}[label=\textnormal{(\roman*)}]
\item On $\mathcal M_t$ for $t>s$, we have
\begin{align*}
|\nabla \mathcal N_s^*|
\le
\left(\frac{3}{2(t-s)}-\scal_{\min}\right)^{1/2},
\qquad
-\frac{3}{2(t-s)}\le \square \mathcal N_s^* \le 0.
\end{align*}

\item For any $0\le s<t^*\le \min\{t_1,t_2\}$, and any $x_1\in \mathcal M_{t_1}$ and $x_2\in \mathcal M_{t_2}$, we have
\begin{align*}
\mathcal N_s^*(x_1)-\mathcal N_s^*(x_2)
\le
\left(\frac{3}{2(t^*-s)}-\scal_{\min}\right)^{1/2}
d_{W_1}^{g_{t^*}}(\nu_{x_1;t^*},\nu_{x_2;t^*})
+\frac{3}{2}\log\left(\frac{t_2-s}{t^*-s}\right).
\end{align*}
\end{enumerate}
\end{thm}

Using the basic properties of the Nash entropy, one can prove the following no-local-collapsing results in the same way as \cite[Theorems 6.1 and 6.2]{bamler2020heatkernel}.

\begin{prop}\label{prop:nolocal}
For any $x_0\in \mathcal M_{t_0}$ with $t_0\ge r^2$, the following statements hold.
\begin{enumerate}[label=\textnormal{(\alph*)}]
\item If $B_{g_{t_0}}(x_0,r)$ is relatively compact in $\mathcal M_{t_0}$ and $\scal\le r^{-2}$ on $B_{g_{t_0}}(x_0,r)$, then
\begin{align*}
|B_{g_{t_0}}(x_0,r)|_{t_0}
\ge
c\exp\bigl(\mathcal N_{x_0}(r^2)\bigr)\,r^3>0,
\end{align*}
where $c$ is a universal constant.

\item If $\scal\ge \scal_{\min}$ on $\mathcal M_{t_0-r^2}$, then
\begin{align*}
|B_{g_{t_0-r^2}}(z,\sqrt{2H_3}\,r)|_{t_0-r^2}
\ge
C(\scal_{\min}r^2)\exp\bigl(\mathcal N_{x_0}(r^2)\bigr)\,r^3>0,
\end{align*}
where $z\in \mathcal M_{t_0-r^2}$ is an $H_3$-center of $x_0$.
\end{enumerate}
\end{prop}

By the same argument, we also obtain the following $\ep$-regularity theorem from \cite[Theorem 10.2]{bamler2020heatkernel}; see also \cite[Theorem 1.6]{heinnaber2014new}.

\begin{thm}\label{thm:epregularity}
There exists a universal constant $\ep>0$ such that for any $x\in \mathcal M_t$ with $t\ge r^2$, if $\mathcal N_x(r^2)\ge -\ep$, then
\begin{align*}
r_{\Rm}(x)\ge \ep r.
\end{align*}
\end{thm}

\subsection{Heat kernel upper and lower bounds}

First, as in the case of closed Ricci flows; cf.\ \cite[Corollaries 9.4, 9.5]{perelman2002entropy}, we have the following lower bound for the heat kernel by using the differential Harnack inequality (see \cite[Corollary 5.8]{lai2021}). Recall that the reduced distance between $x\in \mathcal M_t$ and $y\in \mathcal M_s$, with $t>s\ge 0$, is defined as follows:
\begin{align*}
l_x(y)
=
\frac{1}{2\sqrt{t-s}}
\inf\left\{
\mathcal L(\gamma)\;:\;
\gamma(m)\in \mathcal M_m \text{ for } m\in [s,t],\ \gamma \text{ connects } y \text{ to } x
\right\},
\end{align*}
where
\begin{align*}
\mathcal L(\gamma)
=
\int_s^t \sqrt{t-m}\Bigl(|\pi_S\gamma'(m)|^2+\scal(\gamma(m))\Bigr)\,\d m,
\end{align*}
and $\pi_S\gamma'(m)$ denotes the spatial component of $\gamma'(m)$ in $T\mathcal M_m$. Note that a reduced geodesic exists for such $x$ and $y$, since $\mathcal M$ satisfies the canonical neighborhood assumption and $\scal$ is a proper function on $\mathcal M$; see also \cite[Corollary 5.13]{kleinerlott2020singular2}.

\begin{thm}[Heat kernel estimate I]\label{thm:lower}
For any $x\in \mathcal M_t$ and $y\in \mathcal M_s$ with $t>s\ge 0$, we have
\begin{align*}
K(x;y)\ge \frac{\exp\bigl(-l_x(y)\bigr)}{(4\pi(t-s))^{3/2}}.
\end{align*}
\end{thm}

As an application of Theorem~\ref{thm:lower}, one can prove the following result exactly as in \cite[Proposition 2.21]{fangli2025structure}.

\begin{prop}\label{HncenterScal}
For any $x_0\in \mathcal M_{[0,T]}$ with $t_0:=\mathfrak t(x_0)\ge r^2$ and any constant $R_0>0$, there exists a constant $C=C(T,R_0)>0$ such that the following statements hold.
\begin{enumerate}[label=\textnormal{(\roman*)}]
\item Assume that
\[
|\scal|\le R_0r^{-2}
\]
at the point $x_0(t)\in \mathcal M_t$ for all $t\in [t_0-r^2,t_0]$, where $x_0(t)$ denotes the flow of $x_0$ along its worldline. If $z\in \mathcal M_t$ is an $H_3$-center of $x_0$ with $t\in [t_0-r^2,t_0]$, then
\begin{align*}
d_{g_t}(x_0(t),z)\le C\sqrt{t_0-t}.
\end{align*}

\item Assume that
\[
|\scal|\le R_0r^{-2}
\]
at both $x_0(t)\in \mathcal M_t$ and $y(t)\in \mathcal M_t$ for all $t\in [t_0-r^2,t_0]$. Then for any $t\in [t_0-r^2,t_0]$,
\begin{align*}
d_{g_t}(x_0(t),y(t))
\le
d_{g_{t_0}}(x_0,y(t_0))+C\sqrt{t_0-t}.
\end{align*}
\end{enumerate}
\end{prop}

Next, using Theorem~\ref{thm:nash}, one can follow verbatim the proofs of \cite[Theorems 7.1 and 7.5]{bamler2020heatkernel}; see also \cite[Theorem 10.1]{li2025heatnotes}, to obtain the following heat kernel estimate and its gradient estimate.

\begin{thm}[Heat kernel estimate II]\label{thm:ultra}
For any $x\in \mathcal M_t$ and $y\in \mathcal M_s$ with $t>s\ge 0$, if $\scal\ge \scal_{\min}$ on $\mathcal M_s$, then there exists a constant $C=C(\scal_{\min}(t-s))>0$ such that
\begin{align*}
K(x;y)\le \frac{C}{(t-s)^{3/2}}\exp\bigl(-\mathcal N_x(t-s)\bigr),
\end{align*}
and
\begin{align}\label{eq:gradient}
\frac{|\nabla_x K|(x;y)}{K(x;y)}
\le
\frac{C}{(t-s)^{1/2}}
\sqrt{
\log\left(
\frac{C\exp\bigl(-\mathcal N_x(t-s)\bigr)}{(t-s)^{3/2}K(x;y)}
\right)
}.
\end{align}
\end{thm}

An immediate application of Theorem~\ref{thm:ultra} is the volume non-inflating estimate from \cite[Theorem 8.1]{bamler2020heatkernel}.

\begin{prop}\label{prop:volumenoninflat}
For any $x_0\in \mathcal M_{t_0}$ with $t_0\ge r^2$, if $\scal\ge \scal_{\min}$ on $\mathcal M_{t_0-r^2}$, then for any $L\ge 1$,
\begin{align*}
|B_{g_{t_0}}(x_0,Lr)|_{t_0}
\le
C(\scal_{\min}r^2)\exp\bigl(\mathcal N_{x_0}(r^2)\bigr)\exp(C_0L^2)\,r^3,
\end{align*}
where $C_0$ is a universal constant.
\end{prop}

We now state the following sharp upper bound for the heat kernel. This was proved originally in \cite[Theorem 7.2]{bamler2020heatkernel} for closed Ricci flows, with the coefficient $4+\ep$ replaced by $8+\ep$, and was later improved in \cite[Theorem 2.15(ii)]{fangli2025structure}. Note that the proof carries over verbatim from \cite[Theorem 11.4]{li2025heatnotes}. The key step is to establish an analogue of \cite[Lemma 11.3]{li2025heatnotes}, whose proof is based on a contradiction argument.

\begin{thm}[Heat kernel estimate III]
For any $x\in \mathcal M_t$, $y\in \mathcal M_s$ with $t>s\ge 0$, and any $\ep>0$, if $\scal\ge \scal_{\min}$ on $\mathcal M_s$, then there exists a constant $C=C(\ep,\scal_{\min}(t-s))>0$ such that
\begin{align*}
K(x;y)\le \frac{C\exp\bigl(-\mathcal N_x(t-s)\bigr)}{(t-s)^{3/2}}
\exp\left(-\frac{d_{g_s}^2(z,y)}{(4+\ep)(t-s)}\right),
\end{align*}
where $z\in \mathcal M_s$ is any $H_3$-center of $x$.
\end{thm}

\begin{proof}
We sketch the proof following the presentation in \cite[Theorem 11.4]{li2025heatnotes}.

Fix $y_0\in \mathcal M_{s_0}$. We first claim the following analogue of \cite[Lemma 11.3]{li2025heatnotes}: for any $x\in \mathcal M_t$ with $t\ge s_0$,
\begin{align}\label{eq:heatupper001}
K(x;y_0)\le \bar Q \frac{\exp\bigl(-\mathcal N_{s_0}^*(x)\bigr)}{(t-s_0)^{3/2}}
\exp\left(-\frac{d_{g_{s_0}}^2(z,y_0)}{\bar Q(t-s_0)}\right),
\end{align}
where $z\in \mathcal M_{s_0}$ is any $H_3$-center of $x$, and $\bar Q$ is a constant depending on $\scal_{\min}(t-s_0)$.

Suppose that \eqref{eq:heatupper001} fails for some $x_0\in \mathcal M_{t_0}$. Then, setting $Q_k=2^k\bar Q$ and
\[
t_k:=8^{-k}(t_0-s_0)+s_0,\qquad k\in \mathbb N,
\]
one may apply the point-picking lemma from \cite[Lemma 11.2]{li2025heatnotes} to obtain sequences $x_k\in \mathcal M_{t_k}$ and $z_k\in \mathcal M_{s_0}$ satisfying
\begin{align}\label{eq:heatupper002}
K(x_k;y_0)\ge Q_k \frac{\exp\bigl(-\mathcal N_{s_0}^*(x_k)\bigr)}{(t_k-s_0)^{3/2}}
\exp\left(-\frac{d_{g_{s_0}}^2(z_k,y_0)}{Q_k(t_k-s_0)}\right),
\end{align}
where $z_k$ is an $H_3$-center of $x_k$. Set
\[
d_k:=d_{g_{s_0}}(z_k,y_0),\qquad r_k:=\sqrt{t_k-s_0}.
\]
Then Theorem~\ref{thm:ultra}, together with \eqref{eq:heatupper002}, implies that
\begin{align}\label{eq:heatupper003}
\frac{d_k}{r_k}\ge \sqrt{Q_k(\log Q_k-C_0)}\xrightarrow[k\to\infty]{}+\infty
\end{align}
for some constant $C_0>0$.

For $k$ sufficiently large, we may assume that $r_{\Rm}(y_0)\gg r_k$. For each $k$ and each $t\in [s_0,s_0+r_k^2]$, let $z_k^t\in \mathcal M_t$ be an $H_3$-center of $x_k$. Also, let $y_0(t)\in \mathcal M_t$ denote the point obtained by flowing $y_0$ along its worldline.

\textbf{Claim.} There exists a constant $C_1>0$ such that
\[
d_{g_t}(z_k^t,y_0(t))\ge d_k-C_1r_k.
\]

By Proposition~\ref{HncenterScal}(i), for any $t\in [s_0,s_0+r_k^2]$, the point $y_0$ is an $H$-center of $y_0(t)$ for some constant $H$ independent of $k$ and $t$. By Lemma~\ref{lem:mono1}, this implies
\begin{align*}
d_{g_t}(z_k^t,y_0(t))
&\ge d_{W_1}^{g_t}(\nu_{z_k^t;t},\delta_{y_0(t)})-\sqrt{H_3(t_k-t)} \\
&\ge d_{W_1}^{g_{s_0}}(\nu_{z_k^t;s_0},\nu_{y_0(t);s_0})-\sqrt{H_3(t_k-t)} \\
&\ge d_k-\sqrt{H_3(t_k-s_0)}-\sqrt{H(t-s_0)}-\sqrt{H_3(t_k-t)} \\
&\ge d_k-C_1r_k,
\end{align*}
which proves the claim.

Now define
\[
U_k(t):=B_{g_t}(y_0(t),r_k),\qquad t\in [s_0,s_0+r_k^2].
\]
Then Theorem~\ref{thm:Gaussian} yields
\begin{align}\label{eq:heatupper004}
\int_{U_k(t)} K(x_k;\cdot)\,\d V_{g_t}
&\le C_2 \exp\left(-\frac{(d_{g_t}(z_k^t,y_0(t))-r_k)^2}{5r_k^2}\right) \notag\\
&\le C_3 \exp\left(-\frac{d_k^2}{6r_k^2}\right),
\end{align}
where in the last inequality we used \eqref{eq:heatupper003} together with the claim above.

Since $\square^*K(x_k;\cdot)=0$, it follows from \eqref{eq:heatupper004}, the fact that $r_{\Rm}(y_0)\gg r_k$, and the standard mean value inequality that
\begin{align*}
K(x_k;y_0)\le C_4 \exp\left(-\frac{d_k^2}{6r_k^2}\right).
\end{align*}
This contradicts \eqref{eq:heatupper002}, in view of \eqref{eq:heatupper003}.

Once \eqref{eq:heatupper001} is established, one can follow exactly the same argument as in \cite[Theorem 11.4]{li2025heatnotes} to complete the proof.
\end{proof}

\section{Completion of the singular Ricci flow}

Throughout this section, we consider a singular Ricci flow $\mathcal M$ with a normalized initial condition~$(M,g_0)$. Recall that for any constant $T\ge 1$, there exists a constant $Y=Y(T)>0$ such that $\mathcal M_{[0,T]}$ has entropy bounded below by $-Y$; see Definition~\ref{def:entropybound} and \eqref{entropybd-Y1}.

\subsection{Spacetime completion and conjugate heat kernel measures}

Following \cite[Definition 3.2]{fangli2025structure} and the definition in \eqref{defndstar}, we define a spacetime function $d^*$ on $\MM_{[1,\infty)}$.

\begin{defn}\label{defndstardistance}
For any $x\in \mathcal M_t$ and any $y\in \mathcal M_s$ with $t \ge s \ge 1$, we define
\begin{align*}
d^*(x,y):=\inf_{1\le \tau\le s}
\max\left\{\sqrt{t-\tau},\,
 d_{W_1}^{g_{\tau}}(\nu_{x;\tau},\nu_{y;\tau})\right\}.
\end{align*}
\end{defn}

The $2$-H\"older estimate
\begin{align*}
|\mathfrak t(x)-\mathfrak t(y)|\le d^*(x,y)^2
\qquad \text{for all } x,y\in \MM_{[1,\infty)}
\end{align*}
follows directly from Definition~\ref{defndstardistance}, as in \cite[Proposition 3.9(2)]{fangli2025structure}. Thus, $(\MM_{[1,\infty)},d^*,\mathfrak t)$ is a parabolic metric space; see \cite[Definition 3.19]{fangli2025structure}.

We first prove the following diameter bound.

\begin{lem}
For any $T>1$ and $x,y\in \mathcal M_{[1,T]}$, we have
\begin{align*}
d^*(x,y)\le C,
\end{align*}
where $C$ is a constant depending on $T$ and on the diameter of $(M,g_0)$.
\end{lem}

\begin{proof}
By Definition~\ref{defndstardistance}, for any $x,y\in \mathcal M_{[1,T]}$,
\begin{align*}
d^*(x,y)
&\le \max\left\{\sqrt{T-1},\,d_{W_1}^{g_1}(\nu_{x;1},\nu_{y;1})\right\} \\
&\le \max\left\{\sqrt{T-1},\,\mathrm{diam}_{g_1}(M)\right\} \\
&\le \max\left\{\sqrt{T-1},\,2\mathrm{diam}_{g_0}(M)\right\},
\end{align*}
where, in the last inequality, we used the fact that $\mathrm{diam}_{g_1}(M) \le e^{2\sup_{M \times[0,1]}|\Ric|} \mathrm{diam}_{g_0}(M) \le 2 \mathrm{diam}_{g_0}(M)$ by the normalized condition.
\end{proof}

We now list some basic properties of $d^*$, all of which can be proved verbatim as in the closed Ricci flow case; see \cite[Lemma 3.4, Proposition 3.9, Lemma 3.13, Proposition 3.18]{fangli2025structure}.

\begin{prop}\label{prop:basicproperty}
The function $d^*$ on $\MM_{[1,\infty)}$ satisfies the following properties.
\begin{enumerate}[label=\textnormal{(\arabic*)}]
\item $d^*$ is a distance function on $\MM_{[1,\infty)}$.

\item For any $x,y\in \mathcal M_t$ with $t \ge 1$, we have
\begin{align*}
d^*(x,y)\le d_{g_t}(x,y).
\end{align*}

\item For any $x\in \mathcal M_t$, suppose that $z\in \mathcal M_s$ is an $H$-center of $x$, where $t \ge s \ge 1$. Then
\begin{align*}
d^*(x,z)\le \max\{1,\sqrt H\}\sqrt{t-s}.
\end{align*}

\item For any $T>1$ and $x,y\in \mathcal M_{[1,T]}$, we have
\begin{align*}
|r_{\Rm}(x)-r_{\Rm}(y)|\le C(T)\,d^*(x,y).
\end{align*}
\end{enumerate}
\end{prop}

As in \cite[Definition 3.5]{fangli2025structure}, we denote by $B^*(\cdot,r)$ a metric ball in $\MM_{[1,\infty)}$. The following volume estimates can be proved in the same way as in \cite[Propositions 3.14 and 3.15]{fangli2025structure}, using Propositions~\ref{prop:nolocal} and~\ref{prop:volumenoninflat}.

\begin{prop}\label{prop:volume}
Given $T>1$, for $x\in \mathcal M_{[1,T]}$ and $r>0$ with $\mathfrak t(x)-r^2\ge 1$, the following conclusions hold.
\begin{enumerate}[label=\textnormal{(\roman*)}]
\item For any $t \ge 1$,
\begin{equation*}
|B^*(x,r)\cap \mathcal M_t|_t \le Cr^3,
\end{equation*}
where $C>0$ is a universal constant, and $|\cdot|_t$ denotes the volume with respect to $\d V_{g_t}$.

\item We have
\begin{equation*}
0<c(T)r^5\le |B^*(x,r)|\le Cr^5,
\end{equation*}
where $C>0$ is a universal constant, and $|\cdot|$ denotes the spacetime volume with respect to $\d V_{g_t}\,\d t$.

\item For any $L\ge \sqrt{T}$, we have
\begin{align*}
|B^*(x,L) \cap \mathcal M_{[1,T]}|\le C(T)e^{C_1L^2},
\end{align*}
where $C_1$ is a universal constant.
\end{enumerate}
\end{prop}

Next, we define
\begin{align*}
(Z,d_Z,\mathfrak t)
\end{align*}
to be the metric completion of $\MM_{[1,\infty)}$ with respect to $d^*$. For simplicity, we set $\mathcal R:=\MM_{[1,\infty)}$, which inherits from $\MM_{[1,\infty)}$ the structure of a Ricci flow spacetime $(\mathcal R,\mathfrak t,\partial_{\mathfrak t},g)$, and define $\mathcal S:=Z\setminus \mathcal R$. Thus we obtain the decomposition
\begin{align}\label{eq:regulardecom}
Z=\mathcal R\sqcup \mathcal S.
\end{align}
As above, we denote metric balls in $Z$ with respect to $d_Z$ by $B^*(\cdot,\cdot)$.

As in \cite[Section 5]{fangli2025structure}, one can extend the heat kernel and the conjugate heat kernel measure to points in $\mathcal S$.

\begin{defn}[Conjugate heat kernel measures on $Z$]
For any $x\in \mathcal S$, we define the \textbf{conjugate heat kernel} $K(x;\cdot)$ by
\begin{align}\label{eq:converg}
K(x;\cdot):=\lim_{i\to\infty} K(x_i;\cdot),
\end{align}
where $x_i\in \mathcal R$ converges to $x$ with respect to $d_Z$. The \textbf{conjugate heat kernel measure} $\nu_{x;s}$ based at $x$ is then defined to be the Borel measure on $(\mathcal R_s,g_s)$ given by
\begin{align*}
\d \nu_{x;s}:=K(x;\cdot)\,\d V_{g_s}
\end{align*}
for any $s<\mathfrak t(x)$, and we set $\nu_{x;\mathfrak t(x)}=\delta_x$. Moreover, we define the \textbf{branch} associated with $x$ to be a subset $\mathcal R'\subset \mathcal R_{[1,\mathfrak t(x))}$ such that for each $t\in [1,\mathfrak t(x))$, the slice $\mathcal R_t'$ is a connected component of $\mathcal R_t$ containing the support of $\nu_{x;t}$.
\end{defn}

The existence of the limit in \eqref{eq:converg}, as well as its independence of the approximating sequence $\{x_i\}$, follows from the next lemma. The proof is the same as that of \cite[Corollary 5.18]{fangli2025structure}, based on \eqref{eq:gradient}.

\begin{lem}\label{lem:continuheat}
For any $T>1$ and $x,y\in \mathcal R_{[1,T]}$ with $r=d_Z(x,y)$ and $\max\{\mathfrak t(x),\mathfrak t(y)\}-r^2>1$, and for any $1\le s<s'<\max\{\mathfrak t(x),\mathfrak t(y)\}-r^2$ and any $w\in \mathcal R_s$, we have
\begin{align*}
|K(x;w)-K(y;w)|\le C(T,s'-s)r.
\end{align*}
\end{lem}

By Lemma~\ref{lem:continuheat}, each $\nu_{x;s}$ is a probability measure on $\mathcal R_s$ and satisfies the reproduction formula. Moreover, by passing to the limit, all heat kernel estimates from the previous section remain valid for $K(z;\cdot)$ for any $z\in Z$.

As in Definition~\ref{def:w1}, one can define the $W_1$-Wasserstein distance and the variance between two conjugate heat kernel measures based at $x,y\in Z$. Using Lemma~\ref{lem:continuheat} and taking the limit, one obtains that Lemma~\ref{lem:mono1}, Lemma~\ref{lem:mono2}, Corollary~\ref{cor:concern}, Proposition~\ref{existenceHncenter}, and Theorem~\ref{thm:Gaussian} all extend to general conjugate heat kernel measures based at points of $Z$.

Similarly, one can extend the definitions of $\mathcal N_z$ and $\mathcal W_z$ to points $z\in \mathcal S$. Then all estimates in Subsection~\ref{sub:entropy} continue to hold after passing to the limit.

Next, we define $r_{\Rm}(x)=0$ for every $x\in \mathcal S$, which is justified by the following lemma. For this reason, we call $\mathcal S$ the singular set of $Z$, and \eqref{eq:regulardecom} becomes a regular-singular decomposition.

\begin{lem}\label{lem:curve}
For any $x\in \mathcal S$, suppose that a sequence $x_i\in \mathcal R$ converges to $x$ with respect to $d_Z$. Then
\begin{align}\label{eq:curlimit}
\lim_{i\to\infty} r_{\Rm}(x_i)=0.
\end{align}
\end{lem}

\begin{proof}
By Proposition~\ref{prop:basicproperty}(4), the limit in \eqref{eq:curlimit} exists; denote it by $r_0$. Suppose, for contradiction, that $r_0>0$. Then for all sufficiently large $i$, we have $r_{\Rm}(x_i)\ge r_1:=r_0/2$. By Definition~\ref{def:curvatureradius}, this means that the product domain
\[
P_i:=B_{g_{t_i}}(x_i,r_1)\times [t_i-r_1^2,t_i+r_1^2],
\qquad t_i=\mathfrak t(x_i),
\]
is unscathed and has curvature bounded by $r_1^{-2}$.

For any $s<\mathfrak t(x)$, let $z_s\in \mathcal R_s$ be an $H_3$-center of $x$. By Proposition~\ref{HncenterScal}(i) and Definition~\ref{defndstardistance}, we have
\begin{align*}
\lim_{i\to\infty} d_{g_s}(x_i(s),z_s)=0,
\end{align*}
where $x_i(s)$ denotes the flow of $x_i$ along its worldline to time $s$. It follows that, for large $i$, the point $x_i(s)$ is a $2H_3$-center of $x$. On the other hand, by \cite[Proposition 9.16(b)]{bamler2023compactness}, we have $x\in P_i$ for all sufficiently large $i$. This contradicts the fact that $x\notin \mathcal R$.
\end{proof}

As an immediate corollary of Theorem~\ref{thm:epregularity} and Lemma~\ref{lem:curve}, we obtain the following result, in contrast with Proposition~\ref{prop:nash1}(b).

\begin{lem}
For any $x \in \MS$, we have
\begin{align*}
\NN_x(0):=\lim_{\tau \searrow0} \NN_x(\tau) <0.
\end{align*}
\end{lem}

\subsection{Intrinsic metrics on time-slices}

\begin{defn}\label{defntimeslicedist}
For each $t \ge 1$, we define an extended distance on the time-slice $Z_t$ by
\begin{align*}
d_t^Z(x,y):=\lim_{s\nearrow t} d_{W_1}^{g_s}(\nu_{x;s},\nu_{y;s})\in [0,+\infty]
\end{align*}
for any $x,y\in Z_t$, where the limit exists by Lemma~\ref{lem:mono1}. Note that if $d_t^Z(x,y)<\infty$, then the branches associated with $x$ and $y$ coincide.
\end{defn}

By following the same arguments, we can extend all results from \cite[Section 6]{fangli2025structure} to our setting. We record several key properties below; see \cite[Lemmas 6.3 and 6.4, Propositions 6.6 and 6.8]{fangli2025structure}.

\begin{prop}\label{prop:basic2}
For any $t \ge 1$, the following properties hold.
\begin{enumerate}[label=\textnormal{(\arabic*)}]
\item For any $x,y\in Z_t$,
\begin{align*}
d_Z(x,y)\le d_t^Z(x,y).
\end{align*}

\item $(Z_t,d_t^Z)$ is a complete extended metric space.

\item For any $w\in \mathcal R_t$, there exists a sufficiently small constant $r>0$ such that for any $x,y\in B_{g_t}(w,r)$,
\begin{align*}
d_t^Z(x,y)=d_{g_t}(x,y).
\end{align*}
\end{enumerate}
\end{prop}

In general, the distance $d_t^Z$ on a connected component of $\mathcal R_t$ satisfies $d_t^Z\le d_{g_t}$ and need not coincide with $d_{g_t}$. For noncollapsed Ricci flow limit spaces, \cite[Proposition 6.23]{fangli2025structure} proves that equality holds on each connected component for all but countably many times. In the present singular Ricci flow setting, we will prove (see Corollary~\ref{cor:singulartime} and Proposition~\ref{prop:distance}) that equality holds on each connected component when \(t\) lies outside a subset of Hausdorff \(1/2\)-measure zero.

\subsection{Tangent flows at singular points} \label{subsec:tan}

First, we prove the following result.

\begin{lem}\label{lem:curvatureest}
Given $T>1$ and $t_0 \in [1, T]$, for any $x_0\in Z_{t_0}$ and any $r\in (0,1]$ with $t_0-r^2\ge 1$, suppose that $z\in \mathcal R_{t_0-r^2}$ is an $H_3$-center of $x_0$. Then there exists a point $w\in \mathcal R_{t_0-r^2}$ with
\[
d_{g_{t_0-r^2}}(z,w)\le \sqrt{2H_3}\,r
\]
such that
\begin{align*}
\scal(w)\le C(T)r^{-2}.
\end{align*}
\end{lem}

\begin{proof}
Throughout the proof, all constants $c_i$ and $C_i$ depend only on $T$. By Proposition~\ref{prop:nolocal}(b), we have
\begin{align}\label{eq:curvatureest001}
|B_{g_{t_0-r^2}}(z,\sqrt{2H_3}\,r)|_{t_0-r^2}\ge c_1r^3>0.
\end{align}
Set
\[
B:=B_{g_{t_0-r^2}}(z,\sqrt{2H_3}\,r),
\]
and suppose for contradiction that the conclusion fails.

By the canonical neighborhood assumption, for each point $x\in B$, $(\mathcal R_{t_0-r^2},g_{t_0-r^2},x)$ is $\ep$-close to one of the following: (a) an $\ep$-neck; (b) an $\ep$-cap; or (c) a closed manifold of constant positive sectional curvature. Note that case (c) cannot occur, since otherwise we would have
\[
|B|_{t_0-r^2}\le C\scal(x)^{-3/2},
\]
contradicting \eqref{eq:curvatureest001}. Therefore, under our assumption, $B$ is contained in a subset $B'$ of one of the following types:
\begin{enumerate}[label=\textnormal{(\roman*)}]
    \item $B'$ is an $\ep$-tube, diffeomorphic to $S^2\times \R$ or $S^2\times S^1$;
    \item $B'$ is a capped $\ep$-tube, diffeomorphic to $D^3$ or $\RP^3\setminus \overline{D}^3$;
    \item $B'$ is a doubly capped $\ep$-tube, diffeomorphic to $S^3$, $\RP^3$, or $\RP^3\#\RP^3$.
\end{enumerate}

Set
\[
\inf_{x\in B'} \scal(x)=r_0^{-2}.
\]
By assumption, we have $r_0\ll r$. For any two points $x,y\in B$, Proposition~\ref{prop:geodesicconvex} implies that there exists a minimizing geodesic $\gamma$ connecting $x$ and $y$ with
\[
\mathrm{length}(\gamma)=d_{g_{t_0-r^2}}(x,y)\le 2\sqrt{2H_3}\,r.
\]
Using a covering argument as in Case 2 of the proof of Proposition~\ref{prop:ballpacking}, we obtain
\begin{align*}
|B|_{t_0-r^2}\le |B'|_{t_0-r^2}\le C_2 r_0^2 r,
\end{align*}
which contradicts \eqref{eq:curvatureest001}, since $r_0\ll r$. This completes the proof.
\end{proof}

Next, we introduce the following definition.

\begin{defn}[Tangent flow]\label{def:tangentflow}
For any $x_0\in Z_{t_0}$, a \textbf{tangent flow} at $x_0$ is a Ricci flow obtained as the pointed Cheeger--Gromov limit of
\[
(\mathcal R,r_i^{-2}g_{r_i^2t+t_0},z_i)
\]
for some sequence $r_i\searrow 0$, where $z_i\in \mathcal R_{t_0-r_i^2}$ is an $H_3$-center of $x_0$.
\end{defn}

Given a sequence $r_i\searrow 0$, Lemma~\ref{lem:curvatureest} implies that we can find points
\[
w_i\in B_{g_{t_0-r_i^2}}(z_i,\sqrt{2H_3}\,r_i)
\]
such that
\[
\scal(w_i)\le Cr_i^{-2}.
\]
After passing to a subsequence, there are two possibilities.

\textbf{Case 1.} 
\[
\lim_{i\to\infty}\frac{\scal(w_i)}{r_i^{-2}}=0.
\]

In this case, by the canonical neighborhood assumption (see \cite[Lemma 3.1]{kleinerlott2017singular}), the limiting Ricci flow of
\[
(\mathcal R,r_i^{-2}g_{r_i^2t+t_0},w_i)
\]
is the static Euclidean flow $(\R^3,g_E)$.

\textbf{Case 2.} 
\[
\lim_{i\to\infty}\frac{\scal(w_i)}{r_i^{-2}}=L>0.
\]

In this case, the canonical neighborhood assumption again implies that the limit of
\[
(\mathcal R,r_i^{-2}g_{r_i^2t+t_0},z_i)
\]
is the Ricci flow associated with a Ricci shrinker, since the Nash entropy of the limiting Ricci flow at the base point is constant for all scales. Thus, the Ricci shrinker must be one of
\[
S^3/\Gamma,\qquad S^2\times_{\mathbb Z_2}\R,\qquad S^2\times \R,
\]
each equipped with its standard metric.

Thus the tangent flow in Definition~\ref{def:tangentflow} is well-defined. Moreover, by the above classification, the tangent flow at $x_0$ is independent of the choice of the sequence $\{r_i\}$ and of the $H_3$-centers $z_i$.

By the discussion above and the canonical neighborhood assumption, the following result is immediate; cf.\ \cite[Theorem 7.15]{fangli2025structure}.

\begin{lem}
A point \(x_0\in Z\) lies in \(\RR\) if and only if its tangent flow is the static Euclidean flow.
\end{lem}

For simplicity, as in \cite[Section 5.4]{fangli2025recti}, we write $\mathcal C^0(\Gamma)$, $\mathcal C^1_0(\mathbb Z_2)$, $\mathcal C^1$, and $\mathcal C^3$ for the self-similar ancient Ricci flows whose time-($-1$) slices are the standard $S^3/\Gamma$, $S^2\times_{\mathbb Z_2}\R$, $S^2\times \R$, and $\R^3$, respectively. Note that these spaces are all quotient cylinders in the sense of \cite[Section 5.4]{fangli2025recti}. For each model space, we define the spacetime distance by the variational formula in \cite[Definition 3.2]{fangli2025structure}, and let $\overline{\mathcal C^0(\Gamma)}$, $\overline{\mathcal C^1_0(\mathbb Z_2)}$, $\overline{\mathcal C^1}$, and $\overline{\mathcal C^3}$ denote the corresponding metric completions. Moreover, for each model space, we fix a base point $p^*$ as in \cite[Section 5.4]{fangli2025recti}.

In the following, we use $\mathcal C$ to denote one of $\mathcal C^0(\Gamma)$, $\mathcal C^1_0(\mathbb Z_2)$, $\mathcal C^1$, or $\mathcal C^3$, and use $\overline{\mathcal C}$ to denote its completion. We write $d_{\mathcal C}$ for the corresponding spacetime distance.

\begin{lem}\label{lem:globaltangent}
For any $x_0\in \mathcal S_{t_0}$, suppose that the tangent flow at $x_0$ is given by
\[
\mathcal C\in \{\mathcal C^0(\Gamma),\,\mathcal C^1_0(\mathbb Z_2),\,\mathcal C^1\}.
\]
Then, for any sequence $r_i\searrow 0$, the rescaled spaces
\[
(Z,r_i^{-1}d_Z,x_0,r_i^{-2}(\mathfrak t-t_0))
\]
converge, after passing to a subsequence, in the pointed Gromov--Hausdorff sense to $\overline{\mathcal C}$.
\end{lem}

\begin{proof}
By Proposition~\ref{prop:volume} and a standard ball-packing argument, after passing to a subsequence of $\{r_i\}$, we may assume that
\begin{align}\label{eq:globaltangent001}
(Z,r_i^{-1}d_Z,x_0,r_i^{-2}(\mathfrak t-t_0))
\xrightarrow[i\to\infty]{\mathrm{pGH}}
(Z',d_{Z'},z',\mathfrak t'),
\end{align}
where $(Z',d_{Z'},z',\mathfrak t')$ is a parabolic metric space; see \cite[Definition 3.19]{fangli2025structure}.

By Definition~\ref{def:tangentflow}, the smooth convergence of the corresponding conjugate heat kernels, and Definition~\ref{defndstardistance}, we conclude that there exists a time-preserving isometric embedding
\[
\iota:\mathcal C\to Z'
\]
such that
\[
\iota(\mathcal C)=Z'_{(-\infty,0)}.
\]
Since $Z'$ is complete, this embedding extends uniquely to an isometric embedding
\[
\iota:\overline{\mathcal C}\to Z'.
\]

\textbf{Claim 1.} \ $\iota(\overline{\mathcal C})=Z'_{(-\infty,0]}$.

Let $y'\in Z'_0$. Choose points $y_i\in Z$ converging to $y'$ with respect to the pointed Gromov--Hausdorff convergence in \eqref{eq:globaltangent001}. Fix a small constant $\theta\in (0,1)$, and let $w_i\in \mathcal R_{t_0-\theta r_i^2}$ be an $H_3$-center of $y_i$. After passing to a further subsequence, we may assume that $w_i$ converges to some point $w'\in Z'_{-\theta}$ in the pointed Gromov--Hausdorff sense. Then Proposition~\ref{prop:basicproperty}(3) implies that
\begin{align}\label{eq:globaltangent002}
d_{Z'}(y',w')\le \sqrt{H_3}\sqrt{\theta}.
\end{align}

By the non-expanding estimate (see \cite[Theorem 9.1]{bamler2020structure}), it follows that
\[
d_{g_{t_0-\theta r_i^2}}(w_i,z_i)\le Cr_i,
\]
where $z_i \in \RR_{t_0-\theta r_i^2}$ is an $H_3$-center of $x_0$. Hence, by the convergence in Definition~\ref{def:tangentflow}, there exists a point $w_\infty\in \mathcal C_{-\theta}$ such that
\[
\iota(w_\infty)=w'.
\]
Letting $\theta\searrow 0$, we conclude from \eqref{eq:globaltangent002} that $y'\in \iota(\overline{\mathcal C})$. This proves Claim 1.

\textbf{Claim 2.} \ $Z'_{(0,\infty)}=\emptyset$.

Suppose not. Then there exist $\theta>0$ and a point $y'\in Z'_\theta$. Let $y_i\in Z$ be a sequence converging to $y'$ with respect to the pointed Gromov--Hausdorff convergence in \eqref{eq:globaltangent001}. Let $w_i,z_i\in \mathcal R_{t_0-r_i^2}$ be $H_3$-centers of $y_i$ and $x_0$, respectively. By the same argument as above, we have
\begin{align}\label{eq:globaltangent003}
d_{g_{t_0- r_i^2}}(w_i,z_i)\le Cr_i.
\end{align}

By Lemma~\ref{lem:curvatureest} and the canonical neighborhood assumption, the rescaled flows
\[
(\mathcal R,r_i^{-2}g_{r_i^2t+t_0},w_i)
\]
converge smoothly to a $\kappa$-solution
\[
(M_\infty,g_{\infty,t})_{t\in (-\infty,\theta)}.
\]
On the other hand, \eqref{eq:globaltangent003} and the assumption on the tangent flow at $x_0$ imply that
\[
(M_\infty,g_{\infty,t})_{t\in (-\infty,0)}
\]
is isometric to $\mathcal C$, which is impossible. This contradiction proves Claim 2.

Combining Claims 1 and 2, we conclude that $\iota$ is a time-preserving isometry from $\overline{\mathcal C}$ onto $(Z',d_{Z'},\mathfrak t')$.
\end{proof}

\begin{rem}
By Definition~\ref{def:tangentflow} and Lemma~\ref{lem:globaltangent}, our notion of tangent flow agrees with that in \cite[Definition 7.8]{fangli2025structure}.
\end{rem}

Next, we introduce the following definition of quantitative singular strata, analogous to \cite[Definition 8.11]{fangli2025structure}; see also \cite[Definition 2.22]{bamler2020structure}.

\begin{defn}\label{def:quansing}
For $\ep>0$ and $0<r_1<r_2<\infty$, the quantitative singular strata
\[
\mathcal S^{\ep,0}_{r_1,r_2}\subset \mathcal S^{\ep,1}_{r_1,r_2}\subset Z
\]
are defined as follows. A point $x_0\in \mathcal S^{\ep,k}_{r_1,r_2}$ if and only if
\[
\mathfrak t(x_0)-\ep^{-1}r_2^2\ge 1
\]
and, for every $r\in [r_1,r_2]$, the pointed manifold
\[
(\mathcal R_{\mathfrak t(x_0)-r^2},g_{\mathfrak t(x_0)-r^2},z)
\]
is not $\ep$-close, at scale $r$, to $(\R^3,g_E)$ when $k=1$, and is not $\ep$-close, at scale $r$, to $(S^2\times \R,g_{S^2}\times g_E)$ when $k=0$, where $z\in \mathcal R_{\mathfrak t(x_0)-r^2}$ is any $H_3$-center of $x_0$.
\end{defn}

Note that, by the canonical neighborhood assumption and Lemma~\ref{lem:globaltangent}, Definition~\ref{def:quansing} agrees with \cite[Definition 8.11]{fangli2025structure}. The following identity is immediate from the definitions: for any $L>1$ and $k\in \{0,1\}$,
\begin{align*}
\mathcal S^k=\bigcup_{\ep\in (0,L^{-1})}\bigcap_{0<r<\ep L}\mathcal S^{\ep,k}_{r,\ep L}.
\end{align*}

\subsection{Rectifiability and structure of the singular set} 

By the results of the previous subsection, every singular point in $Z$
is quotient cylindrical; see \cite[Definition 5.24]{fangli2025recti}.
Following essentially the same---and in several places simpler---arguments
as in \cite{bamler2020structure,fangli2025structure,fangli2025loja,
fangli2025recti}, the main structure results for noncollapsed Ricci flow
limit spaces extend to the present singular Ricci flow setting. We briefly
explain the ingredients needed for this extension.

\begin{enumerate}[label=\textnormal{(\arabic*)}]
\item The integral estimates for conjugate heat kernel potentials from
\cite[Proposition 6.2]{bamler2020structure} carry over to singular Ricci
flows. Indeed, the argument is the same as in the proof of Lemma~\ref{usefuliden}: one exhausts each time-slice by compact domains whose
boundary components are central spheres of sufficiently small necks, and
then lets the exhaustion tend to the full time-slice. The boundary terms
vanish by the decay estimate in Lemma~\ref{lem:potent1}. Thus all
integration-by-parts identities used in the closed or limit-space setting
remain valid here.

\item The change-of-base estimate for conjugate heat kernel measures from \cite[Proposition 8.1]{bamler2020structure} (see also \cite[Appendix A]{fangli2025structure}) extends verbatim to singular Ricci flows.

\item The construction of sharp splitting maps from
\cite[Section 3.4]{fangli2025recti} also carries over. As in
\cite[Lemma 3.1]{fangli2025recti}, the splitting map is obtained from the
difference of the heat kernel potential functions based at two suitable
points $x_0$ and $x_1$, where $x_1$ is an independent point relative
to $x_0$. The change-of-base estimate and Lemma~\ref{lem:potent1} give
the required integral control with respect to the conjugate heat kernel
measure based at $x_0$. Consequently, the almost-splitting estimates from
\cite[Appendix C]{fangli2025structure} remain valid in the present setting.

\item The cylindrical Lojasiewicz inequality
\cite[Theorem 1.3]{fangli2025loja} and the corresponding summability
result \cite[Corollary 1.4]{fangli2025loja} near a
$(1,\epsilon,r)$-cylindrical point extend to singular Ricci flows.
Although these results were later formulated for noncollapsed Ricci flow
limit spaces in \cite[Proposition 5.1]{fangli2025recti} using the modified
$\widetilde{\mathcal W}$-entropy, in the present setting one may instead
follow the original proof directly and work with the usual
$\mathcal W$-entropy. The analytic input needed in that proof is supplied
by the heat kernel estimates of Section~\ref{sec:4} and by the
integration-by-parts justification above.

\item The strong uniqueness of cylindrical tangent flows, the
Ahlfors regularity of the packing measures, and the quantitative
stratification estimates used in \cite{fangli2025loja,fangli2025recti} hold for the
completion $Z$. In particular, the neck decomposition theorem and the
corresponding content estimate follow by the same argument as in \cite[Theorem 6.9 and Proposition 6.13]{fangli2025recti}. The present
three-dimensional situation is simpler than the general setting considered
there, because the canonical neighborhood theorem and the classification of
three-dimensional $\kappa$-solutions imply that the only noncompact
singular models are quotient cylindrical. Thus there is no need to consider flat neck regions.
\end{enumerate}

Consequently, the horizontal parabolic rectifiability theorem
\cite[Theorem 1.2]{fangli2025recti}, the vanishing of the
\(1/2\)-dimensional Hausdorff measure of the singular time set analogous to
\cite[Corollary 5.23]{fangli2025recti}, the local $\mathscr H^1$-bound
and time-slice high-curvature volume estimate analogous to
\cite[Theorem 6.26 and Corollary 6.28]{fangli2025recti}, and the
$L^1$-curvature estimate analogous to
\cite[Theorem 6.31]{fangli2025recti} all hold in the present singular
Ricci flow setting. We state the consequences needed below.

\begin{thm}[{\cite[Theorem 1.2]{fangli2025recti}}]\label{thm:rec}
For each $k=0,1$, the set $\mathcal S^k$ is horizontally parabolic $k$-rectifiable with respect to the distance $d_Z$. In particular, $\mathcal S$ is horizontally parabolic $1$-rectifiable.
\end{thm}

The definition of horizontally parabolic $k$-rectifiable can be found in \cite[Definition 1.1]{fangli2025recti}. As an application of Theorem~\ref{thm:rec}, one obtains the following result, analogous to \cite[Corollary 5.23]{fangli2025recti}, which improves \cite[Theorem 1.4]{kleinerlott2020singular2}.

\begin{cor}\label{cor:singulartime}
We have
\begin{align*}
\mathscr H^{1/2}(\mathfrak t(\mathcal S))=0,
\end{align*}
where $\mathscr H^{1/2}$ denotes the \(1/2\)-dimensional Hausdorff measure on $\R$.
\end{cor}

\begin{proof}
Let $d_P$ denote the parabolic distance on $\R$, namely
\[
d_P(s,t)=\sqrt{|s-t|}
\qquad \text{for all } s,t\in \R.
\]
It suffices to prove that
\begin{align}\label{eq:para0}
\mathscr H_P^1(\mathfrak t(\mathcal S_{[1,T]}))=0,
\end{align}
for any $T>1$, where $\mathscr H_P^1$ denotes the $1$-dimensional Hausdorff measure on $\R$ with respect to the metric $d_P$.

As in the proof of \cite[Corollary 5.23]{fangli2025recti}, for sufficiently small $\delta>0$, we have
\begin{align*}
\mathcal S_{[1,T]}\setminus \mathcal S^0
\subset
\bigcup_i \bigl(\CCC_{i,0}\cap B^*(x_i,r_i)\bigr),
\end{align*}
where
\[
\NNN_i 
=
B^*(x_i,2r_i)\setminus B^*_{r_x}(\CCC_i)
\]
is a $(1,\delta, \cc, r_i)$-neck region with $x_i\in \mathcal S_{[1,T]}$, where $\cc=10^{-100}$; cf.\ \cite[Definition 5.9]{fangli2025recti}. Moreover, by the definition of a cylindrical neck region, we have
\[
\CCC_{i,0}\cap B^*(x_i,2r_i)\subset \mathcal S.
\]

We claim that
\begin{align}\label{eq:connected002}
\mathscr H_P^1\bigl(\mathfrak t(\mathcal S_{[1,T]}\setminus \mathcal S^0)\bigr)=0.
\end{align}
To prove \eqref{eq:connected002}, it suffices to show that for each $i$,
\begin{align}\label{eq:connected003}
\mathscr H_P^1\bigl(\mathfrak t(\CCC_{i,0}\cap B^*(x_i,r_i))\bigr)=0.
\end{align}

Let
\[
A_i:=\CCC_{i,0}\cap \overline{B^*(x_i,r_i)},
\]
which is a closed subset of $\mathcal S$. By Ahlfors regularity (see \cite[Theorem 5.21]{fangli2025recti}), we have
\begin{align*}
\mathscr H^1(A_i)
\le
\mathscr H^1\bigl(\CCC_{i,0}\cap B^*(x_i,3r_i/2)\bigr)
\le
C(T)r_i<\infty.
\end{align*}
Since $A_i$ is compact, for any $\ep>0$ there exists $r_\ep>0$ such that every $x\in A_i$ is uniformly $(1,\ep,r_\ep)$-cylindrical; see \cite[Lemma 5.10]{fangli2025recti}. Hence
\begin{align}\label{eq:connected005}
|\mathfrak t(x)-\mathfrak t(y)|
\le
\Psi(\ep)\,d_Z^2(x,y)
\end{align}
for any $x,y\in A_i$ with $d_Z(x,y)\le r_\ep$.

By the definition of Hausdorff measure, there exists a countable cover $\{B^*(y_j,s_j)\}$ of $A_i$ with $s_j<r_\ep/2$ such that
\begin{align}\label{eq:connected006}
\sum_j s_j \le C(T)r_i.
\end{align}
Then, by \eqref{eq:connected005}, for any $x,y\in B^*(y_j,s_j)$ we have
\begin{align*}
\sqrt{|\mathfrak t(x)-\mathfrak t(y)|}
\le
\Psi(\ep)\,d_Z(x,y)
\le
\Psi(\ep)\,s_j.
\end{align*}
Therefore,
\begin{align*}
\mathscr H_P^1\bigl(\mathfrak t(B^*(y_j,s_j)\cap A_i)\bigr)\le \Psi(\ep)\,s_j.
\end{align*}
Summing over $j$ and using \eqref{eq:connected006}, we obtain
\begin{align*}
\mathscr H_P^1(\mathfrak t(A_i))\le \Psi(\ep)\,r_i.
\end{align*}
Since $\ep>0$ is arbitrary, it follows that
\[
\mathscr H_P^1(\mathfrak t(A_i))=0,
\]
which proves \eqref{eq:connected003}, and hence \eqref{eq:connected002}.

On the other hand, since $\mathscr H^1(\mathcal S^0)=0$ and the time function $\mathfrak t$ is Lipschitz with respect to $d_Z$ and $d_P$, we also have
\[
\mathscr H_P^1(\mathfrak t(\mathcal S^0))=0.
\]
Combining this with \eqref{eq:connected002}, we obtain \eqref{eq:para0}. This completes the proof.
\end{proof}

Next, we state the following neck decomposition theorem, analogous to \cite[Theorem 6.9]{fangli2025recti}, which generalizes Theorem~\ref{neckdecomgeneral3d}.

\begin{thm}[3D neck decomposition theorem II]\label{neckdecomgeneral3dsing}
For any constants $T>1$, $\delta>0$ and $\eta>0$, if $\zeta \le \zeta(T, \delta, \eta)$, then the following holds.

Given $z_0 \in Z_{[1,T]}$ with $\t(z_0)-2 \zeta^{-2} r_0^2 >1$, we have the decomposition\textup{:}
	\begin{align*}
		&B^*(z_0,r_0)\subset \bigcup_a\big(\NNN_a\bigcap B^*(x_a,r_a)\big)\bigcup \bigcup_b B^*(x_b,r_b)\bigcup S^{1,\delta,\eta},\\
		&S^{1,\delta,\eta}\subset \bigcup_a\big(\CCC_{0,a}\bigcap B^*(x_a,r_a)\big)\bigcup\tilde{S}^{1,\delta,\eta},
	\end{align*}
with the following properties\textup{:}
	\begin{enumerate}[label=\textnormal{(\alph{*})}]
		\item For each $a$, $\NNN_a=B^*(x_a,2r_a)\setminus B^*_{r_x}(\CCC_a)$ is a $(1,\delta, \cc, r_a)$-cylindrical neck region, where $\cc=\cc(T)$. 
		
		\item For each $b$, there exists a point in $B^*(x_b,2 r_b)$ which is $(2,\eta,r_b)$-symmetric.
		
		\item The following content estimates hold\textup{:}
		\begin{align*}
\sum_a r_a+\sum_b r_b+\HHH^1(S^{1,\delta,\eta})\leq C(T) r_0 \quad \text{and} \quad \HHH^1(\tilde{S}^{1,\delta,\eta})=0.
	\end{align*}	
	\end{enumerate}
\end{thm}

As an application of Theorem~\ref{neckdecomgeneral3dsing}, we obtain the following results, analogous to \cite[Theorem 6.26 and Corollary 6.28]{fangli2025recti}.

\begin{thm}
For any $T>1$ and $z_0\in Z_{[1,T]}$ with $\mathfrak t(z_0)-2r_0^2>1$, the following statements hold.
\begin{enumerate}[label=\textnormal{(\roman*)}]
\item We have
\begin{align*}
\mathscr H^1\bigl(\mathcal S\cap B^*(z_0,r_0)\bigr)\le C(T)r^5_0.
\end{align*}

\item For any $0<r<1$ and any $t \in [1, T]$, we have
\begin{align*}
\bigl|\{r_{\Rm}<rr_0\}\cap B^*(z_0,r_0)\cap Z_t\bigr|_t \le C(T)r^2r_0^3.
\end{align*}
\end{enumerate}
\end{thm}

\subsection{Curvature, volume, and diameter estimates}

As another application of Theorem~\ref{neckdecomgeneral3dsing}, we have the following $L^1$-curvature estimate, analogous to \cite[Theorem 6.31]{fangli2025recti}, which improves \cite[Theorem 1.1]{kleinerlott2020singular2}.

\begin{thm}\label{thm:RmL1y}
For any $T>1$, there exists a constant $C>0$, depending on $T$ and the diameter of $(M,g_0)$, such that for any $t \in [2, T]$, we have
\begin{align*}
\int_{\mathcal R_t} |\Rm|\,\d V_{g_t}
\le
\int_{\mathcal R_t} r_{\Rm}^{-2}\,\d V_{g_t}
\le C.
\end{align*}
\end{thm}

As an immediate corollary of Theorem~\ref{thm:RmL1y}, together with \cite[Proposition 5.5(3)]{kleinerlott2017singular}, we obtain the following Lipschitz continuity of the total volume, which improves \cite[Theorem 1.2]{kleinerlott2020singular2}.

\begin{cor}
For any $T>1$, there exists a constant $C>0$, depending on $T$ and the diameter of $(M,g_0)$, such that for any \(s,t\in[2,T]\), we have
\begin{align*}
\bigl||\mathcal R_s|_s-|\mathcal R_t|_t\bigr|\le C|s-t|.
\end{align*}
\end{cor}

Next, we prove the following result, which may be regarded as a generalized form of Perelman's bounded diameter conjecture.

\begin{thm}\label{thm:bdddiam}
Let $D_{g_t}$ denote the sum of the diameters, with respect to $d_{g_t}$, of all connected components of $\mathcal R_t$. Then for any $T>2$,
\begin{align*}
D_{g_t}\le C
\end{align*}
for any $t \in [2, T]$, where $C>0$ is a constant depending on $T$ and the diameter of $(M,g_0)$.
\end{thm}

\begin{proof}
Fix $t \in [2, T]$. Define a measure $\mu_t$ on $\mathcal R_t$ by
\begin{align*}
\mu_t(A):=\int_A (|\scal|+1)\,\d V_{g_t}
\end{align*}
for any Borel subset $A\subset \mathcal R_t$. Suppose that $\mathcal R_t'$ is a connected component of $\mathcal R_t$ whose diameter with respect to $d_{g_t}$ is $D_{g_t}'$. By Proposition~\ref{prop:geodesicconvex}, there exists a minimizing geodesic $\gamma$ in $\mathcal R_t'$ of length $D_{g_t}'/2$.

For a small universal constant $\delta>0$, the canonical neighborhood assumption implies that there exists $r_\delta>0$ depending on $T$ such that for any $x\in \mathcal R_t$ with $\scal(x)\ge r_\delta^{-2}$, the pointed manifold $(\mathcal R_t,g_t,x)$ is $\delta$-close to a $\kappa$-solution. Set $r_0:=\theta r_\delta$, where $\theta\ll 1$ is a universal constant.

Choose a covering $\{B_{g_t}(x_i,2r_0)\}_{i=1}^N$ of $\gamma$ such that $x_i\in \gamma$, the balls $\{B_{g_t}(x_i,r_0)\}_{i=1}^N$ are pairwise disjoint, and
\begin{align}\label{eq:bdddiam001}
C_0^{-1}Nr_0\le D_{g_t}'\le C_0Nr_0,
\end{align}
where $C_0>1$ is universal. We may assume that there exists $N'\in [0,N]$ such that for each $i\in [1,N']$ ($[1, N']=\emptyset$ if $N'=0$), there is a point $y_i\in B_{g_t}(x_i,r_0/2)$ with $\scal(y_i)<r_\delta^{-2}$, while for each $i\in [N'+1,N]$, we have $\scal\ge r_\delta^{-2}$ on $B_{g_t}(x_i,r_0/2)$. As in the proof of Proposition~\ref{prop:ballpacking}, we have
\begin{align*}
\mu_t(B_{g_t}(x_i,r_0))\ge c_1r_0^3
\end{align*}
for $i\in [1,N']$, and
\begin{align*}
\mu_t(B_{g_t}(x_i,r_0))\ge c_2r_0
\end{align*}
for $i\in [N'+1,N]$, where $c_1,c_2>0$ are universal constants. Hence, by \eqref{eq:bdddiam001},
\begin{align}\label{eq:bdddiam002}
\mu_t(\mathcal R_t')
\ge
\sum_{i=1}^N \mu_t(B_{g_t}(x_i,r_0))
\ge
c_1N'r_0^3+c_2(N-N')r_0
\ge
c_3D_{g_t}',
\end{align}
where $c_3>0$ is a constant depending on $T$.

On the other hand, by Theorem~\ref{thm:RmL1y} and \cite[Proposition 5.5(1)]{kleinerlott2017singular}, the total mass $\mu_t(\mathcal R_t)$ is uniformly bounded by a constant depending only on $T$ and the diameter of $(M,g_0)$. The conclusion now follows from \eqref{eq:bdddiam002}.
\end{proof}

As an application, we prove the following corollary.

\begin{cor}\label{cor:noncompact}
For any $t \ge 1$, if $\mathcal R_t$ has a noncompact component, then $\mathcal S_t\neq \emptyset$.
\end{cor}

\begin{proof}
Let $\mathcal R_t'$ be a noncompact connected component of $\mathcal R_t$. Choose a sequence $x_i\in \mathcal R_t'$ such that $\scal(x_i)\to +\infty$. Since $d_{g_t}$ is bounded on $\mathcal R_t'$ by Theorem~\ref{thm:bdddiam}, after passing to a subsequence, we may assume that $\{x_i\}$ is a Cauchy sequence with respect to $d_{g_t}$.

By Proposition~\ref{prop:basic2}(1), the sequence $\{x_i\}$ is also Cauchy with respect to $d_Z$. Hence, $x_i$ converges to some point $x_\infty\in Z_t$. Since $\scal(x_i)\to +\infty$, Proposition~\ref{prop:basicproperty}(4) implies that $x_\infty\notin \mathcal R_t$. Therefore, $x_\infty\in \mathcal S_t$, and in particular $\mathcal S_t\neq \emptyset$.
\end{proof}

\subsection{Gromov--Hausdorff convergence of time-slices}

For each $t_0 >1$, let $Z_{t_0}'$ be a connected component of $Z_{t_0}$. We first introduce the following definition.

\begin{defn}\label{defn:componentbranch}
A branch $\RR' \subset \RR_{[1, t_0)}$ is \textbf{determined} by $Z_{t_0}'$ if for every $z\in Z_{t_0}'$, the support of $\nu_{z;s}$ is contained in $\RR_s'$ for all $s\in[1,t_0)$.
\end{defn}

It is clear from Definition~\ref{defntimeslicedist} that $Z_{t_0}'$ determines a unique branch $\RR'$. Moreover, $Z_{t_0}'$ also determines a branch $Z'\subset Z_{[1,t_0)}$ such that $Z_s'$ is connected for every $s\in[1,t_0)$ and contains $\RR_s'$.

As before, let $(\overline{\mathcal R_t'},d_{g_t})$ denote the metric completion of $(\mathcal R_t',d_{g_t})$. By the same argument as in the proof of Theorem~\ref{thm:convergence}, we obtain the following Gromov--Hausdorff convergence result.

\begin{prop}\label{prop:ztstructure}
In the above setting, we have the Gromov--Hausdorff convergence
\[
    (\overline{\mathcal R_t'},d_{g_t})
    \xrightarrow[t\nearrow t_0]{\GHconvtext}
    (Z'_{t_0},d^Z_{t_0}).
\]
In particular, \((Z'_{t_0},d^Z_{t_0})\) is a compact geodesic space.
The last conclusion follows from Proposition~\ref{prop:geodesicconvex} and the Gromov--Hausdorff convergence.
\end{prop}

As a corollary, we prove the following.

\begin{prop}\label{prop:distance}
For any $t_0 \in [1, \infty) \setminus \t(\MS)$, the distance $d^Z_{t_0}$ agrees with $d_{g_{t_0}}$ on each connected component of $\RR_{t_0}$.
\end{prop}

\begin{proof}
Since $t_0 \notin \t(\MS)$, the slice $Z_{t_0}$ contains no singular points. By Corollary~\ref{cor:noncompact}, every connected component of $\RR_{t_0}$ is compact. Hence each connected component of $\RR_{t_0}$ is also a connected component of $Z_{t_0}$, because by Proposition~\ref{prop:ztstructure} every connected component of $Z_{t_0}$ is a compact geodesic space.

Now let $\RR'_{t_0}$ be a connected component of $\RR_{t_0}$, and let $x,y\in \RR'_{t_0}$. Since $(\RR'_{t_0},d^Z_{t_0})$ is a geodesic space, there exists a minimizing geodesic
\[
\gamma\subset \RR'_{t_0}
\]
with respect to $d^Z_{t_0}$ connecting $x$ and $y$ and hence $d_{g_{t_0}}(x,y) \le d^Z_{t_0}(x,y)$. On the other hand, we always have
\[
d^Z_{t_0}(x,y)\le d_{g_{t_0}}(x,y)
\]
as a consequence of Proposition~\ref{prop:basic2}(3). Therefore,
\[
d^Z_{t_0}(x,y)=d_{g_{t_0}}(x,y).
\]
This proves the proposition.
\end{proof}

Next, we prove the following past-continuity statement for the intrinsic time-slices.

\begin{prop}\label{prop:pastcont}
In the above setting, we have the Gromov--Hausdorff convergence
\[
    (Z'_t,d^Z_t)
    \xrightarrow[t\nearrow t_0]{\GHconvtext}
    (Z'_{t_0},d^Z_{t_0}).
\]
\end{prop}

\begin{proof}
Let \(t_j\nearrow t_0\) be an arbitrary sequence. By Proposition~\ref{prop:ztstructure}, for each \(j\) we may choose \(s_j<t_j\) sufficiently close to \(t_j\) such that \(s_j\nearrow t_0\) and
\[
    d_{\mathrm{GH}}\left(
        (Z'_{t_j},d^Z_{t_j}),
        (\overline{\mathcal R'_{s_j}},d_{g_{s_j}})
    \right)
    \le 2^{-j}.
\]
On the other hand, applying Proposition~\ref{prop:ztstructure} again to the time \(t_0\), we have
\[
    (\overline{\mathcal R'_{s_j}},d_{g_{s_j}})
    \xrightarrow[j\to\infty]{\GHconvtext}
    (Z'_{t_0},d^Z_{t_0}).
\]
Combining these two facts, we conclude that
\[
    (Z'_{t_j},d^Z_{t_j})
    \xrightarrow[j\to\infty]{\GHconvtext}
    (Z'_{t_0},d^Z_{t_0}).
\]
Since the sequence \(t_j\nearrow t_0\) was arbitrary, the proof is complete.
\end{proof}

Finally, by the same argument, we generalize Theorem~\ref{thm:singular} to the present setting.

\begin{thm}
For any $T>1$, there exists a constant $C>0$, depending on $T$ and the diameter of $(M,g_0)$, such that the following properties hold for any $t\in [2,T]$:
\begin{enumerate}[label=\textnormal{(\arabic*)}]
\item For any $r>0$,
\begin{equation*}
\bigl|\{y\in Z_t \mid d_t^Z(y,\mathcal S_t)<r\}\bigr|_t\le Cr^2.
\end{equation*}

\item The Minkowski dimension of $\mathcal S_t$ with respect to $d_t^Z$ satisfies
\begin{equation*}
\dim_{\mathscr M}(\mathcal S_t)\le 1.
\end{equation*}
\end{enumerate}
\end{thm}

\section{Higher-dimensional extensions under positive isotropic curvature}
The main results of this paper can be extended to higher-dimensional Ricci flows under the positive isotropic curvature condition.

Suppose that~$(M^n,g_0)$ is a normalized Riemannian manifold with positive isotropic curvature (PIC) and contains no nontrivial incompressible $(n-1)$-dimensional space form. In addition, we assume that either $n=4$ or $n\ge 12$. We also note the preprint
\cite{chen2024}, which states an improvement for $n \ge 9$.

It follows from \cite{chenzhu2006surgery} and \cite{brendle2019iso} that one can construct a Ricci flow with surgery starting from~$(M^n,g_0)$. As in the three-dimensional case, the canonical neighborhood assumption holds: points of high scalar curvature have local geometry modeled on $\kappa$-solutions, where a $\kappa$-solution is a $\kappa$-noncollapsed complete ancient Ricci flow solution with uniformly positive isotropic curvature. The classification of all such $\kappa$-solutions is now complete; see \cite{brendlenaff2023,brendledaskalopoulosnaffsesum2023,choli2023}.

By taking a limit of Ricci flows with surgery as the surgery parameter tends to zero, one obtains, as in \cite{kleinerlott2017singular}, a singular Ricci flow $\mathcal M$ starting from~$(M^n,g_0)$. As in dimension three, this singular Ricci flow is unique; see \cite{haslhofer2022}.

On this singular Ricci flow, one can follow essentially the same arguments as in the previous sections to obtain analogous results. The main difference is that the only noncompact tangent flow is the standard Ricci flow on $S^{n-1}\times \R$. Accordingly, one needs to work with $(1,\delta,\mathfrak c,r)$-cylindrical neck regions and the corresponding neck decomposition theorem.

We now list some of the main results.

\begin{thm}
Let $\MM$ be the singular Ricci flow starting from $(M^n, g_0)$. Let $(Z, d_Z, \t)$ be the completion of $\MM_{[1,\infty)}$ with respect to the spacetime distance $d^*$. Then the following statements hold.
\begin{enumerate}[label=\textnormal{(\arabic*)}]
\item The singular set $\MS$ is horizontally parabolic $1$-rectifiable with $\mathscr H^{1/2}(\t(\MS))=0$.

\item For any $T>1$ and $z_0\in Z_{[1, T]}$ with $\mathfrak t(z_0)-2r_0^2>1$, we have
\begin{align*}
\mathscr H^1\bigl(\mathcal S\cap B^*(z_0,r_0)\bigr)\le C(n, T)r_0^{n+2}.
\end{align*}

\item For any $T>2$, there exists a constant $C>0$, depending on $n$, $T$ and the diameter of $(M,g_0)$, such that for any $t \in [2, T]$, we have
\begin{align*}
D_{g_t}+\int_{\mathcal R_t} |\Rm|^{\frac{n-1}{2}}\,\d V_{g_t}
\le D_{g_t}+
\int_{\mathcal R_t} r_{\Rm}^{-(n-1)}\,\d V_{g_t}
\le C,
\end{align*}
where $D_{g_t}$ denotes the sum of the diameters, with respect to $d_{g_t}$, of all connected components of $\mathcal R_t$.

\item For each $t_0>1$, we have the Gromov--Hausdorff convergence
\begin{align*}
(\overline{\mathcal R_t'},d_{g_t})
\xrightarrow[t\nearrow t_0]{\GHconvtext}
(Z_{t_0}',d_{t_0}^Z),
\end{align*}
where $Z'_{t_0}$ is a connected component of $Z_{t_0}$ that determines $\RR_t'$ for $t < t_0$. In particular, $(Z_{t_0}',d_{t_0}^Z)$ is a compact geodesic space.

\item For any $t \ge 1$, the Minkowski dimension of $\mathcal S_t$ with respect to $d_t^Z$ satisfies
\begin{equation*}
\dim_{\mathscr M}(\mathcal S_t)\le 1.
\end{equation*}
\end{enumerate}
\end{thm}

\begingroup
\small
\setlength{\labelsep}{0.35em}
\bibliographystyle{alpha_nobreak}
\bibliography{bibfile}
\endgroup

\vspace{10pt}

Yu Li, Institute of Geometry and Physics, University of Science and Technology of China, No. 96 Jinzhai Road, Hefei, Anhui Province, 230026, China; Hefei National Laboratory, No. 5099 West Wangjiang Road, Hefei, Anhui Province, 230088, China; Email: yuli21@ustc.edu.cn.

%\end{CJK}
\end{document}